\documentclass[12pt, dvips, twoside]{amsart}
\usepackage{amsmath,amssymb,amscd,amsxtra,latexsym,graphicx,epsfig,
euscript,amsthm}
\usepackage[dvips]{color}

\theoremstyle{plain}
\newtheorem{theorem}{Theorem}[section]
\newtheorem*{theorem*}{Theorem}
\newtheorem{lemma}[theorem]{Lemma}
\newtheorem{proposition}[theorem]{Proposition}

\newtheorem*{remarks*}{Remarks}

\newtheorem*{example*}{Example}
\newtheorem*{examples*}{Examples}
\newtheorem{definition}[theorem]{Definition}
\newtheorem*{definition*}{Definition}

\newtheorem{remark}{Remark}

\newcommand{\Pmn}{{\mathcal P}^{\mu}_{\nu^*}(C_k(r))}

\newcommand{\Pmpn}{{\mathcal P}^{\mu+1}_{\nu^*}(C_k(r))}
\newcommand{\Fmn}{{\rm P}^{\mu}_{\nu^*}(C_k(r))}

\newcommand{\Fer}{{\rm P}}
\newcommand{\QFer}{{\rm Q}}
\newcommand{\PP}{{\mathcal P}}

\newcommand{\R}{{\mathbb R}}

\newcommand{\C}{{\mathbb C}}

\def\1{\:\!}
\def\2{\;\!}

\def\Diffc0{\operatorname{Diff^c_0}}

\def\Sympc0{\operatorname{Symp^c_0}}

\def\const{\operatorname{const}}

\def\gl{\lambda}

\def\.{\mskip1mu}
\def\?{\mskip-1mu}

\newcounter{mnotecount}[section]

\newcommand{\rmnote}[1]{}

\title[Delaunay type domains for an overdetermined elliptic problem]{Delaunay type domains for an overdetermined elliptic problem in $\mathbb{S}^n \times \mathbb{R}$ and $\mathbb{H}^n \times \mathbb{R}$}

\author{Filippo Morabito, Pieralberto Sicbaldi}
\address{Filippo Morabito}
\address{Department of Mathematics,
Korea University, Anam-dong, 136-701, Seoul, South Korea}
\address{\&}
\address{Korea Institute for Advanced Study,
School of Mathematics, Seoul, South Korea}
\email{morabitf@gmail.com}
\address{Pieralberto Sicbaldi}
\address{Laboratoire d'Analyse Topologie Probabilit\'es,
Universit\'e Aix-Marseille 3, Avenue de l'Escadrille Normandie
Niemen, 13397 Marseille, cedex 20, France}
\email{pieralberto.sicbaldi@univ-cezanne.fr}

\begin{document}

\begin{abstract}
We prove the existence of a countable family of Delaunay type domains
\[
\Omega_j \subset \mathbb{M}^n \times \mathbb{R}
\]
where $\mathbb{M}^n$ is the Riemannian manifold $\mathbb{S}^n$ or
$\mathbb{H}^n$ and $n \geq 2$, bifurcating from the cylinder
$B^{n} \times \mathbb{R}$ (where $B^{n}$ is a geodesic ball of radius 1 in
$\mathbb{M}^n$) for which the first eigenfunction of the
Laplace-Beltrami operator with zero Dirichlet boundary condition also
has constant Neumann data at the boundary: for each~$n \in
\mathbb{N}\setminus \{0\} $ and $j=1/n$, the overdetermined system
$$
\left\{
\begin{array} {ll}
\Delta_g\, u + \gl\, u = 0 &\mbox{in }\; \Omega_j\\
u=0 & \mbox{on }\; \partial \Omega_j \\
g(\nabla u, \nu) = \const &\mbox{on }\; \partial \Omega_j
\end{array}
\right.
$$
has a bounded positive solution for some positive constant
$\lambda$, where $g$ is the usual metric in $\mathbb{M}^n \times
\mathbb{R}$. The domains $\Omega_j$ are rotationally symmetric and
periodic with respect to the $\mathbb{R}$-axis of the cylinder and
as $j \to 0$ the domain $\Omega_j$ converges to the cylinder
$B^{n} \times \mathbb{R}$.
\end{abstract}

\maketitle


\section{Introduction and statement of the result}

A long-standing open problem is to classify domains $\Omega \subseteq \mathbb{R}^{n}$, $n \geq 2$, for which the over-determined elliptic problem
\begin{equation}
\label{problem.intro}
\left\{
\begin{array}{rclll}
    \displaystyle \Delta u + f(u) & = & 0 & \textnormal{in} & \Omega \\[3mm]
    \displaystyle u& >& 0 & \textnormal{in} & \Omega \\[3mm]
    \displaystyle u & = & 0 & \textnormal{on} & \partial \Omega \\[3mm]
    \displaystyle \langle \nabla  u , \nu\rangle & = & \textnormal{constant} & \textnormal{on} & \partial\Omega  \, ,
\end{array}
\right.
\end{equation}
has a solution $u \in C^{2}(\overline\Omega)$, where $f$ is a given Lipschitz function,
 $\nu$ is the normal vector about $\partial \Omega$, and
  $\langle \cdot,\cdot \rangle$ denotes the usual scalar product.

\medskip

In \cite{Serrin}, J. Serrin proved that if $\Omega$ is a $C^2$ bounded
domain where there exists a solution
$u \in C^2(\overline\Omega)$ to problem (\ref{problem.intro}) and $f$ is $C^1$, then
$\Omega$ must be a ball. The same result can be obtained when $f$ is
supposed to have only Lipschitz regularity, see \cite{Puc-Ser}. The
result by J. Serrin has been of outstanding importance for two
reasons: first, because it made available to the mathematical
analysis community the moving plane method, introduced some years
before by A. D. Alexandrov to prove that the only compact and
constant mean curvature hypersurfaces embedded in $\mathbb{R}^n$ are
the spheres (see~\cite{Alex}), and secondly for the many applications
to physics and to applied mathematics. Indeed, problem
(\ref{problem.intro}), when $f$ is constant, describes a viscous
incompressible fluid moving in straight parallel streamlines through
a straight pipe of given cross sectional form $\Omega$. If we fix
rectangular coordinates $(x,y,z)$ with the $z$-axis directed along
the pipe, it is well known that the flow velocity $u$ along the pipe
is then a function of $x$ and $y$, and satisfies
\[
\Delta u + k = 0
\]
where $k$ is a constant related to the viscosity and density of the
fluid. The adherence condition is given by $u=0$ on $\partial
\Omega$. The result by J. Serrin allows us to state that the
tangential stress per unit area on the pipe wall (represented by
$\mu\, \langle \nabla u, \nu \rangle$, where $\mu$ is the viscosity)
is the same at all points of the wall if and only if the pipe has a
circular cross section. An other model referable to problem
(\ref{problem.intro}), is the linear theory of torsion of a solid
straight bar of cross section $\Omega$, see for example \cite{Sok},
and in this framework the result by J. Serrin states that when a
solid straight bar is subject to torsion, the magnitude of the
resulting traction which occurs at the surface of the bar is
independent of the position if and only if the bar has a circular
cross section.

\medskip

Overdetermined boundary conditions arise naturally also in free
boundary problems, when the variational structure imposes suitable
conditions on the separation interface: see for example
\cite{Alt-Caf}. In this context it is important to underline that
several methods for studying locally the regularity of solutions of
free boundary problems are often based on blow-up techniques applied
to the intersection of $\Omega$ with a small ball centered in a
point of $\partial \Omega$, which lead then to the study of an
elliptic problem in an unbounded domain. In this framework, problem
(\ref{problem.intro}) in unbounded domains was considered by H.
Berestycki, L. Caffarelli and L. Nirenberg in \cite{BCN}. In this
case, the typical nonlinearity $f$ taken into account was $f(u) =
u-u^3$, which reduces the equation in (\ref{problem.intro}) to the
Allen-Cahn equation. Under the assumptions that $\Omega$ is a
Lipschitz epigraph with some suitable control at infinity for its
boundary, they proved that if problem (\ref{problem.intro}) admits a
smooth, bounded solution, then $\Omega$ is a half-space.

\medskip

Some rigidity results for overdetermined elliptic problems in
$\mathbb{R}^2$  and $\mathbb{R}^3$ have been obtained in some recent
papers by A. Farina and E. Valdinoci. In particular they obtain
natural assumptions under which one can conclude that if $\Omega$ is
an epigraph where there exists a solution to problem
(\ref{problem.intro}) then $\Omega$ must be a half-space and $u$ is
a function of only one variable,
see~\cite{Far-Val-0},~\cite{Far-Val-2} and~\cite{Far-Val-1}.

\medskip

For some types of functions $f$ the structure of the family of
domains $\Omega$ where the overdetermined problem
(\ref{problem.intro}) can be solved shares many similarities with
the class of embedded constant mean curvatures hypersurfaces. For
the bounded case, the analogy is very simple: the only compact
embedded constant mean curvature hypersufaces in $\mathbb{R}^n$ are
the spheres by the result of A.D. Alexandrov \cite{Alex} and the
only bounded domains in $\mathbb{R}^n$ where problem
\eqref{problem.intro} can be solved are the balls by the result of
J. Serrin \cite{Serrin}. In \cite{S} P. Sicbaldi showed the
existence of perturbations of the right solid cylinder in
$\mathbb{R}^n$ for $n \geq 3$, rotationally symmetric and periodic
in the vertical direction, where it is possible to solve problem
(\ref{problem.intro}) for the function $f(t) = \lambda\, t$. In
\cite{SS}, F. Schlenk and P. Sicbaldi show that previous examples
belong in fact to a smooth 1-parameter family of unbounded domains
$s \mapsto \Omega_s$ where \eqref{problem.intro} can be solved,
whose boundaries are smooth periodic hypersurfaces of revolution
with respect to an $\mathbb{R}$-axis. They obtain such a result by
paralleling the construction of the well known Delaunay's family of
constant mean curvature surfaces.

\medskip

It is interesting to remark that domains where problem
(\ref{problem.intro})  with $f = 0$ can be solved arise as limits
under scaling of sequences of domains where problem
(\ref{problem.intro}) with $f(t) = \lambda\, t$ can be solved, just
like minimal surfaces arise as limits under scaling of sequences of
constant mean curvature surfaces. Problem (\ref{problem.intro}) in
the interesting case when $f = 0$ has been recently studied in
$\mathbb{R}^2$ by F. H\'elein, L. Hauswirth and F. Pacard, in a work
highly inspired by the theory of minimal surfaces, where they get a
 Weierstrass representation for domains where the overdetermined
problem can be solved. In a recent paper, \cite{traizet}, M. Traizet
proves a one-to-one correspondence between 2-dimensional domains
where problem (\ref{problem.intro}) with $f = 0$ can be solved and a
special class of minimal surfaces, showing that the link is in fact
very deep.

\medskip

In \cite{Ros-Sic}, A. Ros and P. Sicbaldi studied problem
(\ref{problem.intro})  obtaining geometric and topological necessary
properties of the domain $\Omega$ and leading to geometric and
topological conditions for the existence of a solution of
(\ref{problem.intro}). Their work is highly inspired by the theory
of constant mean curvature surfaces embedded in $\mathbb{R}^3$, and
in particular they obtain, for 2-dimensional domains where
(\ref{problem.intro}) can be solved, a kind of half-space theorem
and for some functions $f$ also a result on the boundedness of the
ends of the domain. More precisely they prove the following two results:
\begin{enumerate}
\item If $\Omega$ is contained in a half-space of $\mathbb{R}^2$
and problem (\ref{problem.intro}) can be solved with $|\nabla u|$ bounded,
then $\Omega$ is either a ball, or a half-plane, or it is contained
in a slab and is symmetric with respect to the axis of the slab.
\item If $\Omega$ has finite topology and problem (\ref{problem.intro})
can be solved for some super-linear function $f$ (i.e. there exists a
positive constant $\lambda$ such that $f(t) \geq \lambda\, t$), then
each end $E$ of $\Omega$ is contained in a half-slab.
\end{enumerate}

\medskip

The strong link between overdetermined elliptic problems and
constant mean curvatures surfaces lets us think that the structure
of the family of domains where (\ref{problem.intro}) can be solved
is very rich and interesting. As for constant mean curvature
surfaces, overdetermined problems can be considered also in a
Riemannian manifold, and in this framework problem
(\ref{problem.intro}) becomes

\begin{equation}
\label{problem.intro1}
\left\{
\begin{array}{rclll}
    \displaystyle \Delta_g u + f(u) & = & 0 & \textnormal{in} & \Omega \\[3mm]
    \displaystyle u& >& 0 & \textnormal{in} & \Omega \\[3mm]
    \displaystyle u & = & 0 & \textnormal{on} & \partial \Omega \\[3mm]
    \displaystyle g( \nabla  u , \nu) & = & \textnormal{constant} & \textnormal{on} & \partial\Omega  \, ,
\end{array}
\right.
\end{equation}
where $g$ denotes the metric of the manifold and $\Delta_g$ is the
Laplace-Beltrami operator.

\medskip

\medskip

In this paper we are interested in generalize the result of P.
Sicbaldi  in \cite{S} to manifolds $\mathbb{S}^n \times \mathbb{R}$
and $\mathbb{H}^n \times \mathbb{R}$, where $\mathbb{S}^n$ and
$\mathbb{H}^n$ are respectively the n-dimensional manifolds of
constant sectional curvature equal to 1 or -1. In fact, we prove
that the solid right   cylinder $B^n \times \mathbb{R}$ (where $B^n$
is a geodesic ball in $\mathbb{S}^n$ or $\mathbb{H}^n$) can be
perturbed in order to obtain new domains where problem
(\ref{problem.intro1}) can be solved for the function $f(t) =
\lambda\, t$ for some positive constant $\lambda$. The boundary of
such domains is rotationally symmetric with respect to the
$\mathbb{R}$-axis of the cylinder, and is periodic in the vertical
direction. We can refer to such domains as Delaunay type domains.

\medskip

In order to state our result, let us denote by $\mathbb{M}^n$  the
Riemannian manifold $\mathbb{S}^n$ or $\mathbb{H}^n$. We denote the
coordinates of $\mathbb{M}^n \times \mathbb{R}$ as $(x,t)$, $x \in
\mathbb{M}^n$ and $t \in \mathbb{R}$. Let us fix a point $0$
(origin) in $\mathbb{M}^n$ and denote by $\|x\|$ the distance of $x
\in \mathbb{M}^n$
 to the origin $0 \in \mathbb{M}^n$. Our main result is the
following:

\begin{theorem}
\label{maintheorem} There exist a real positive number $T_*$, a
sequence of real  positive numbers $T_j \longrightarrow T_*$ and a
sequence of nonconstant functions $v_j \in C^{2,\alpha}(\mathbb{R})$ of period $T_j$ converging to $0$ in $C^{2,\alpha}(\mathbb{R})$ such that the domains
\[
\Omega_j = \left\{(x,t) \in \mathbb{M}^n \times \mathbb{R} \, , \,
\|x\| < 1 + v_j(t)\right\}
\]
have a positive solution $u_j \in C^{2,\alpha}(\Omega_j)$ to
the problem \eqref{problem.intro}. Moreover
$\displaystyle \int_0^{T_j} v_j\, dt = 0$.
\end{theorem}

We remark that we do not have a smooth one-parameter family of domains, but only a sequence of domains converging to the straight cylinder. According with the case of $\mathbb{R}^n$ in \cite{SS}, it is tempting to conjecture that such domains belong in fact to a smooth one-parameter family of domains.

\medskip

\section{Rephrasing the problem}

Given a continuous function $v : \mathbb{R}/T
\mathbb{Z} \longmapsto (0, +\infty)$ whose $L^\infty$-norm is small enough, we define
\[
C_{1+v}^T  : =  \left\{ (x,t) \in \mathbb{M}^{n}\times \mathbb{R}/T
\mathbb{Z} \, : \, 0 \leq \|x\|   < 1 + v(t) \right\} \, .
\]
Our aim is to show that there exists a positive real number $T_*$, a sequence $T_j \to T_*$ and a sequence of nonconstant
functions $v_j \in C^{2,\alpha}(\mathbb{R}/T_j \mathbb{Z})$ of mean equal to zero and converging to the zero function in the $C^{2,\alpha}$-norm, such that the overdetermined
problem
\begin{equation}\label{a1}
\left\{
\begin{array}{rcccl}
    \Delta_{g} \, \phi + \lambda \, \phi & = & 0 & \textnormal{in} & C_{1+v}^T \\[3mm]
    \phi & = & 0 & \textnormal{on} & \partial C_{1+v}^T \\[3mm]
    \displaystyle   g( \nabla  \phi ,  \nu) & = & {\rm constant} & \textnormal{on} & \partial C_{1+v}^T
\end{array}
\right.
\end{equation}
has a nontrivial positive solution $(\phi, \lambda) = (\phi_j, \lambda_j)$ for the sequence $(v_j, T_j)$. Here $\nu$ denote the normal
vector field about $\partial C_{1+v}^T$, $\lambda$ is a positive
constant, and $g$ is the product metric of $\mathbb{M}^n \times
\mathbb{R}/T\, \mathbb{Z}$ (in particular the second factor is equipped with the metric induced by the standard metric of $\R$).

\medskip

We remark that the symmetry
of the problem allow us to require the function $v$ to be even.
\medskip

Let us denote by $g_{\mathbb{M}^n}$ the usual metric on
$\mathbb{M}^n$. Let $\lambda_{1}$ be the first eigenvalue of the
Laplace-Beltrami operator with zero Dirichlet boundary condition in
the unit geodesic ball
\[
B_1 = \{x \in \mathbb{M}^n \, : \, \|x\| < 1\} \,.
\]
We denote by
$\tilde \phi_{1}$ the associated eigenfunction
\begin{equation}
\left\{
\begin{array}{rcccl}
    \Delta_{g_{\mathbb{M}^n}} \tilde \phi_{1} + \lambda_{1}\, \tilde \phi_{1} & = & 0 & \textnormal{in} & B_{1} \\[3mm]
    \tilde \phi_{1} & = & 0 & \textnormal{on} & \partial B_{1}
\end{array}
\right.
\end{equation}
which is normalized to have $L^2(B_1)$-norm equal to $1/2\pi$. Then $\phi_1(x,t) =
\tilde \phi_1(x)$ solve the problem
\begin{equation}
\left\{
\begin{array}{rcccl}
    \Delta_{g } \phi_{1} + \lambda_{1}\, \phi_{1} & = & 0 & \textnormal{in} & C_{1}^T \\[3mm]
    \phi_{1} & = & 0 & \textnormal{on} & \partial C_{1}^T
\end{array}
\right.
\label{eq:11-11}
\end{equation}
and
\begin{equation}
\label{norm1} \displaystyle \int_{C_1^{2\pi}} \phi_1^2\, {\rm
dvol}_{g } = 1.
\end{equation}
As $\phi_1$ do not depend on $t$, sometimes we will write simply
$\phi_1(x)$.

 \medskip

Let $C^{2,\alpha}_{\textnormal{even},0} (\mathbb{R}/2\pi
\mathbb{Z})$  be the set of even functions on $\mathbb{R}/2\pi
\mathbb{Z}$ of mean equal to zero. For all function $v \in C^{2,\alpha}_{\textnormal{even},0}
(\mathbb{R}/2\pi \mathbb{Z})$ whose norm is small enough, the
domain $C_{1+v}^T$ is well defined for all $T >0$ and standard results
on Dirichlet eigenvalue problem (see \cite{GT} and \cite{C}) apply
to give the existence, for all $T >0$, of a unique positive function
\[
\phi = \phi_{v,T} \in C^{2, \alpha} \left(C_{1+v}^T\right)
\]
and a constant $\lambda = \lambda_{v,T} \in \mathbb R$ such that
$\phi$ is  a  solution to the problem
\begin{equation}\label{formula}
\left\{
\begin{array}{rcccl}
    \Delta_{g} \, \phi + \lambda \,\phi & = & 0 & \textnormal{in} & C_{1+v}^T \\[3mm]
    \phi & = & 0 & \textnormal{on} & \partial C_{1+v}^T
\end{array}
\right.
\end{equation}
which is normalized by
\begin{equation}
\displaystyle \int_{C_{1+v}^{2\pi}} \left(\phi\left(x,\frac{T}{2\pi}\,t\right)\right)^2\, {\rm dvol}_{g} = 1
\label{noral}
\end{equation}
In addition $\phi$ and $\lambda$ depend smoothly on the function
$v$, and $\phi = \phi_1$, $\lambda =\lambda_1$ when $v\equiv 0$.

\medskip

After canonical identification of $\partial C_{1+v}^T$ with
$S^{n-1}\times \mathbb{R}/T \mathbb{Z}$, we define the
Dirichlet-to-Neumann operator $N$~:
\[
N(v,T)  =  \displaystyle  g (\nabla \phi ,  \nu )  \, |_{\partial
C_{1+v}^{T}}   - \frac{1}{{\rm Vol}_{g} (\partial C_{1+v}^{T})} \,
\int_{\partial C_{1+v}^{T}} \, g (\nabla \phi ,  \nu ) \,
\mbox{dvol}_{g} \, ,
\]
where $\nu$ denotes the unit normal vector field to  $\partial
C_{1+v}^{T}$ and {$\phi$ is the solution of (\ref{formula}). A
priori $N(v,t)$ is a function defined over $S^{n-1}\times
\mathbb{R}/T \mathbb{Z}$, but it is easy to see that it depends only
on the variable $t \in \mathbb{R}/T \mathbb{Z}$ because $v$ has such
a property. For the same reason it is an even function, and moreover
it is clear that its mean vanishes. If now we operate a rescaling
and we define
\begin{equation}
\label{operator.F}
F(v,T)\, (t)\, = N(v,T)\, \left(\frac{T}{2\pi}\,t\right) \, ,
\end{equation}
 Schauder's estimates imply that {$F$} is well defined in a neighborhood of $(0,T)$ in the space $\mathcal
\mathcal{C}^{2,\alpha}_{\textnormal{even},0} (\mathbb{R}/2\pi \mathbb{Z})
\times \mathbb{R}$, and takes its values in $C^{1,\alpha}_{\textnormal{even},0} (\mathbb{R}/2\pi \mathbb{Z})$.
Our aim is to find a positive real number $T_*$, a sequence $T_j \to T_*$ and a sequence of nonconstant functions $v_j \in C^{2,\alpha}(\mathbb{R}/T_j \mathbb{Z})$ of mean equal to zero and converging to the zero function in the $C^{2,\alpha}$-norm, such that $F(v_j,T_j)=0$. Observe that, with
this condition, $\phi = \phi_{v_j,T_j}$ will be the solution to the problem
\eqref{a1} and our main theorem \ref{maintheorem} will be proved.

\section{The linearized operator}
\label{linearized}

Let $k$ the sectional curvature of the manifold $\mathbb{M}^n$ (i.e.
$k=1$ if $\mathbb{M}^n = \mathbb{S}^n$ and $k=-1$ if $\mathbb{M}^n =
\mathbb{H}^n$). If we choose spherical coordinates $(r, \theta)$,
with $\theta \in S^{n-1}$ and $r\in[0, +\infty)$ if $k<0$ and $r \in
[0,\pi]$ if $k
>0$, the usual metric in $\mathbb{M}^n$ can be written as
\[
g_{\mathbb{M}^n} = \textnormal{d}r^2 + S_k(r)^2 \,
\textnormal{d}\theta^2
\]
where
\[
S_k(r) = \left\{
\begin{array}{lll}
\sinh r & \textnormal{if} & k=-1\\
 \sin r & \textnormal{if} & k=1
\end{array}\right.
\]
(see \cite{C}, Section II.5, Theorem 1).

\begin{remark}\label{remark_k} {\rm In fact our result is true also in a more general situation. Instead of considering $\mathbb{S}^n$ and $\mathbb{H}^n$ whose sectional curvatures are $1$ and $-1$, it would be possible to consider the $n$-dimensional manifold $\mathbb{M}^n(k)$ of constant sectional curvature $k \neq 0$. When $k>0$ we have to assume that $\mathbb{M}^n(k)$ contains a geodesic ball of radius 1, and this is not true for all $k >0$. Using spherical coordinates $(r, \theta)$, with $\theta \in
S^{n-1}$ and $r\in[0, +\infty)$ if $k<0$ and $r \in
[0,\pi/\sqrt{k})$ if $k
>0$, the usual metric in $\mathbb{M}^n(k)$ is
\[
g_{\mathbb{M}^n(k)} = \textnormal{d}r^2 + S_k(r)^2 \,
\textnormal{d}\theta^2
\]
where
\[
S_k(r) = \left\{
\begin{array}{lll}
\displaystyle \frac{1}{\sqrt{-k}}\, \sinh (\sqrt{-k}\,r) & \textnormal{if} & k<0\\[3mm]
\displaystyle \frac{1}{\sqrt{k}}\, \sin (\sqrt{k}\,r)  & \textnormal{if} & k>0
\end{array}\right.
\]
(see \cite{C}, Section II.5, Theorem 1). The computations that follow are true in general for the manifold $\mathbb{M}^n(k) \times \mathbb{R}$, under the hypothesis that $\mathbb{M}^n(k)$ contains a geodesic ball $B_1$ of radius 1, even though we consider only the cases $k=1$ and $-1$.}
\end{remark}

\medskip

For all $ v \in \mathcal
\mathcal{C}^{2,\alpha}_{\textnormal{even},0} (\mathbb{R}/2\pi
\mathbb{Z})$ and all $T > 0$ let $\psi$ be the (unique) solution
(periodic with respect to the variable $t$) of
\begin{equation}\label{c00}
\left\{
\begin{array}{rcll}
    \displaystyle \Delta_{g} \psi + \lambda_{1}\,\psi & = & 0  & \textnormal{in} \qquad C_1^T \\[3mm]
    \psi & = &  - \displaystyle  {\partial_r \phi_{1}} \, v(2 \pi t/T) &
     \textnormal{on} \qquad \partial C_1^T
\end{array}
\right.
\end{equation}
which is $L^2(C_1^T)$-orthogonal to $\phi_1$.
We define
\begin{equation}
\label{h11}
\tilde H_T(v) : = \left.\left( {\partial_r \psi }
+ {\partial^2_r \phi_{1}} \, v\left(\frac{2 \pi t}{T}\right) \right) \right|_{\partial C_1^T}.
\end{equation}
By symmetry it is clear that $\tilde H_T(v)$ is a function only depending on $t$, then changing the variable we can define
\begin{equation}
H_T(v)(t) : = \tilde H_T(v) \left(\frac{T}{2 \pi}\, t \right) .
\label{Hache}
\end{equation}

The main result of this section is the:

\begin{proposition}
The linearization of the operator $F$ with respect to $v$ computed at the point $(0,T)$ is given by $H_T$.
\end{proposition}
\begin{proof} To linearize the operator  $F$ (see \eqref{operator.F})  with respect
 to $v$ at $(0,T)$ we will compute
\[
\lim_{s \rightarrow 0} \frac{F(s\, w, T) - F(0, T)}{s} .
\]
Precisely we determine the first order approximation of $F(s\, w,T)$
with respect to the variable $s$.
 We
parameterize $C_{1+v}^T,$ with $v=sw,$ on $C_{1}^{2 \pi}$ by
\[
Y ( y,t) : =\left( \left(1 + s \, \chi ( y) \,  w   \right)  \, y\,
, \, \frac{T t}{2\pi}\right)
\]
for $y \in \mathbb{M}^n$ and $t \in \mathbb{R}$ and where $\chi$ is
a cutoff function identically equal to $0$ when $\|y\| \leq 1/4$ and
identically equal to $1$ when $\|y\|\geq 1/2$. The metric induced by
$Y$ will be denoted by
\begin{equation}
\hat g : = Y^*  g \label{hat.metric}
\end{equation}
If $\phi$  solves
\begin{equation}
\label{system.0} \left\{
\begin{array}{rcccl}
    \Delta_{g} \,  \phi + \lambda \,  \phi & = & 0 &
    \textnormal{in} & C_{1+v}^{2\pi} \\[3mm]
    \phi & = & 0 & \textnormal{on} & \partial C_{1+v}^{2\pi}
\end{array}
\right.
\end{equation}
with
\[
\int_{C_{1+v}^{2\pi}}  \phi^2 \, \mbox{dvol}_{g} =1,
\]
then  $\hat \phi = Y^* \phi$   is solution (smoothly depending on
the real parameter $s$) to
\begin{equation}
\label{system.hat} \left\{
\begin{array}{rcccl}
    \Delta_{\hat{g}} \, \hat \phi + \hat \lambda \, \hat \phi & = & 0 & \textnormal{in} & C_1^{2\pi} \\[3mm]
    \hat \phi & = & 0 & \textnormal{on} & \partial C_1^{2\pi}
\end{array}
\right.
\end{equation}
with $\hat \lambda = \lambda $ and satisfying
\[
\int_{C_1^{2\pi}} \hat \phi^2 \, \mbox{dvol}_{\hat g} =1.
\]
We remind that the function $\phi_1$ is the positive solution to
\eqref{system.0} with $v=0$, $\lambda=\lambda_1$ and $L^2$-norm equal to 1. Clearly the
function $\hat \phi_1  : = Y^* \phi_1$ solves \eqref{system.hat}.
 Moreover
\begin{equation}
\label{boundary} \hat \phi_1 (y,t) = \phi_1 \left(   \left(1+ s \, w
\right) \, y\, , \frac{T t}{2\pi}
 \right)
\end{equation}
on $\partial C_1^{2\pi}$. Writing $\hat \phi = \hat \phi_1 + \hat
\psi$ and $\hat \lambda = \lambda_1 + \mu$, we find out that $\hat
\psi$ solves
\begin{equation}
\label{f001} \left\{
\begin{array}{rcccl}
    \Delta_{\hat{g}} \, \hat \psi + (\lambda_1 + \mu) \, \hat \psi + \mu \, \hat \phi_1 & = & 0 & \textnormal{in} & C_1^{2\pi} \\[3mm]
    \hat \psi & = & - \hat \phi_1 & \textnormal{on} & \partial C_1^{2\pi}
\end{array}
\right.
\end{equation}
with
\begin{equation}
\label{formula-new-10} \int_{C_1^{2\pi}} (2 \, \hat  \phi_1 \, \hat
\psi+  \hat \psi^2  ) \,  \mbox{dvol}_{\hat g} = \int_{C_1^{2\pi}}
\phi_1^2 \, {\rm dvol}_{g} -  \int_{C^{2\pi}_{1+s w}} \phi_1^2 \,
{\rm dvol}_{ g}.
\end{equation}
Obviously $\hat \psi$ and $\mu$ are smooth functions of $s$. If $s
=0$, then $C_{1+sw}^{T}=C_{1}^{T}$ and \eqref{system.hat} reduces to
\eqref{system.0} with $v=0.$ In particular we have $ \phi
=\phi_1=\hat \phi_1,$  $\lambda = \lambda_1,$
 $\hat  \psi \equiv 0,$ $\mu=0$ and $\hat g = g$. We set
\[
\dot \psi := \partial_s \hat \psi |_{s=0} \qquad \mbox{and} \qquad
\dot \mu := \partial_s \mu |_{s=0}.
\]
Differentiating (\ref{f001}) with respect to $s$ and evaluating  the
result at $s=0$, we obtain
\begin{equation}
\label{f00100} \left\{
\begin{array}{rlllll}
    \Delta_{ g} \, \dot \psi + \lambda_1 \, \dot \psi + \dot \mu \, \phi_1  & = & 0 & \textnormal{in} & C_1^{2\pi} \\[3mm]
    \dot \psi & = & - \partial_r \phi_1 \,  w & \textnormal{on} & \partial C_1^{2\pi}
\end{array}
\right.
\end{equation}
because from \eqref{boundary}, differentiation with respect to $s$
at $s=0$ yields $\partial_s \hat \phi_1 |_{s=0} = \partial_r \phi_1
\, w$, where $r=\|y\|.$
\medskip

Differentiating \eqref{formula-new-10} with respect to $s$ and
evaluating the result at $s=0$, we obtain
\begin{equation}
\label{formula-new-100} \int_{C_1^{2\pi}}  \phi_1 \, \dot \psi  \,
\mbox{dvol}_{ g} = 0.
\end{equation}
Indeed, the derivative of the right hand side of
\eqref{formula-new-10} with respect to $s$ vanishes when $s=0$ since
$\phi_1$ vanishes identically on $\partial C_1^{2\pi}$.

\medskip

{If we multiply the first equation of \eqref{f00100} by $\phi_1$ and
we integrate it over $C_1^{2\pi}$, using the boundary condition and
the fact that the average of $w$ is $0$ we conclude that $\dot \mu
=0$.} Consequently the $2 \pi$-periodic function $\dot \psi(y,t)$
is related to the solution $\psi(y,t)$ of \eqref{c00} by the
identity $\psi(y,t):=\dot \psi(y,2\pi t/T).$
We proved that
\[
\hat \phi(y,t) = \hat \phi_1(y,t) + s \,  \psi(y,T t/2 \pi) +
{\mathcal O} (s^2).
\]
In particular, in $C_1^{2\pi} \setminus C_{3/4}^{2\pi}$, we have
\[
\begin{array}{rlll}
\hat \phi  (y,t) & = & \displaystyle \phi_1 \left( \left(1+ s \, w\right)
\, y\,,\, T t/2\pi \right) + s \,  \psi (y,T t/2 \pi) + {\mathcal O} (s^2) \\[3mm]
& =  & \displaystyle \phi_1 \left(y,T t/2\pi\right) +
s \left(  w \, r \, \partial_r \phi_1 + \psi(y,T t/2 \pi) \right)
+{\mathcal O} (s^2).
\end{array}
\]

\medskip

To complete the proof of the result, we will  compute $\hat g(\nabla
\hat \phi,\hat \nu)$ on the boundary of $C^{2\pi}_1.$
Such a function is the normal
derivative of $\hat \phi$ when the normal is computed
with respect to the metric $\hat g$. We use the coordinates
$(r,\theta, t),$ with $y = r \, \theta,$ (being $(r,\theta)$ the
coordinates introduced at the beginning of Section \ref{linearized}).
In $C_1^{2\pi}\setminus C_{3/4}^{2\pi}$ the metric $\hat g$ equals
\begin{eqnarray*}
\hat g = (1 + s w)^2  dr^2 + s r\, w'\, (1 + s w)\, dr\, dt
 + \left((T/2\pi)^2 +  s^2\, r^2\, (w')^2 \right)\, dt^2
 + S_k(r)^2\, (1+sw)^2\, d \theta^2.
\end{eqnarray*}
It follows from this
expression that the unit normal vector field to $\partial
C_1^{2\pi}$ for the metric $\hat g$ is given by
\[
\hat \nu  = \left( ( 1 + s \, w )^{-1} +  {\mathcal O} (s^2)\right)
\, \partial_r +  {\mathcal O} (s) \, \partial_{t}.
\]
As a result, on $\partial C_1^{2\pi},$
\[
\hat g ( \nabla  \hat \phi , \hat \nu ) =  \partial_r \phi_1  + s \,
\left(w\, \,  \partial_r^2 \phi_1  + \partial_r  \psi(y,Tt/2 \pi)
\right) + {\mathcal O}(s^2).
\]
On $\partial C_1^{2\pi}$ the term $ w  \,  \partial_r^2 \phi_1  +
\partial_r \psi(y,Tt/2 \pi) $ has mean equal to zero and $\partial_r \phi_1$
is constant. Using $\hat g ( \nabla  \hat \phi , \hat \nu )$
to compute $F(sw),$   we get that the linearized of
$F$ is $H_T$. \hfill$\Box$
\end{proof}

\section{The structure of the linearized operator}

Let $v \in \mathcal \mathcal{C}^{2,\alpha}_{\textnormal{even},0} (\mathbb{R}/2\pi \mathbb{Z})$. Recalling that the mean of $v$ is zero and the fact that $v$ is even, by Fourier expansion $v$ can be written as
\begin{equation}
\label{fourier} v = \sum_{j \geq 1} a_j\, \cos(jt).
\end{equation}

Observe that in principle $\phi_1$ is only defined in the cylindrical domain $C_1^{2\pi}$,
however, this function being radial in the first $n$ variables and not depending on $t$, it is a solution of a second order ordinary differential equation and then it can be extended at least in a neighborhood of $\partial C_1^{2\pi}$.
\medskip

We will need the following:
\begin{lemma}
\label{le:3.1}
Assume that $v \in \mathcal \mathcal{C}^{2,\alpha}_{\textnormal{even},0} (\mathbb{R}/2\pi \mathbb{Z})$ and write $v$ as in (\ref{fourier}). For $T >0$ we define
\[
\phi_{0} (x,t) =  \partial_r \phi_{1}(x) \, v (2 \pi t /T)
\]
where $r = \|x\|$. Then
\begin{equation}\label{b}
\Delta_{{g}} \phi_{0} + \lambda_{1} \, \phi_{0} = {\partial_r
\phi_{1}} \, \sum_{j \geq 1} a_j\, \frac{1}{S_k(r)^{2}} \,
\cos\left(\frac{2\pi j t}{T}\right)\, \left[n-1 - \left(\frac{2 \pi
j}{T}\right)^2 S_k(r)^2 \right].
\end{equation}
\end{lemma}
{\bf Proof :} The Laplace-Beltrami operator for the metric $g$ can be written as
\[
\Delta_g = \partial_r^2 + (n-1)\frac{C_k(r)}{S_k(r)}\, \partial_r + \frac{1}{S_k(r)^2}\, \Delta_{S^{n-1}} + \partial _t^2
\]
where
\[
C_k(r) = \left\{
\begin{array}{lll}
\cosh r & \textnormal{if} & k=-1\\
\cos r & \textnormal{if} & k=1
\end{array}\right.
\]
(see \cite{C}, Section II.5, Theorem 1). Then it is easy to compute
\[
\Delta_{g} \, {\partial_r \phi_{1}}  =
 - \lambda_{1} \, {\partial_r \phi_{1}} + \frac{n-1}{S_k^2(r)} \, {\partial_r \phi_{1}}  \,
\]
and
\begin{eqnarray*}
    \Delta_{{g}} \phi_{0} & = & - \lambda_{1} \, \phi_{0} + {\partial_r \phi_{1}} \, \sum_{j \geq 1} a_j\, \frac{1}{S_k(r)^{2}} \, \cos\left(\frac{2\pi j t}{T}\right)\, \left[ n-1 - \left(\frac{2 \pi j}{T}\right)^2 S_k(r)^2
    \right].
\end{eqnarray*}
This completes the proof of the result. \hfill $\Box$

\begin{remark}\label{remark_k2} {\rm With respect to Remark \ref{remark_k} we give the formula of $C_k(r)$ when $k \notin \{-1,1\}$. In fact, we have
\[
C_k(r) = \left\{
\begin{array}{lll}
\cosh (\sqrt{-k}\, r )& \textnormal{if} & k<0\\
\cos (\sqrt{k}\, r ) & \textnormal{if} & k>0 \, .
\end{array}\right.
\]
}
\end{remark}

\medskip

We investigate now the structure of the linearized operator $H_T$. The main result of this section is the:

\begin{proposition}
\label{H}
For all $T>0$, the operator
\[
H_T : \mathcal \mathcal{C}^{2,\alpha}_{\textnormal{even},0} (\mathbb{R}/2\pi \mathbb{Z}) \longrightarrow \mathcal C^{1,\alpha}_{\textnormal{even},0} (\mathbb{R}/2\pi \mathbb{Z}) ,
\]
is a self adjoint, first order elliptic operator preserving, for all $j \in \mathbb{N} \backslash \{0\}$, the eigenspace $V_j$ spanned by the function $\cos(jt)$.
\end{proposition}

{\bf Proof.}  The fact that $H_T$ is a first order elliptic operator is standard since it is the sum of the Dirichlet-to-Neumann operator for $\Delta_{g} + \lambda_1$ and  a constant times the identity. In particular, elliptic estimates yield
\[
\| H_T(w)  \|_{\mathcal C^{1,\alpha}_{\textnormal{even},0}
(\mathbb{R}/2\pi \mathbb{Z})} \leq c \, \| w \|
_{\mathcal \mathcal{C}^{2,\alpha}_{\textnormal{even},0}
(\mathbb{R}/2\pi \mathbb{Z})}.
\]
The fact that the operator $H_T$ is (formally) self-adjoint is easy. Let $\psi_1$ (resp. $\psi_2$) the solution of \eqref{c00} corresponding to the function ${w_{1}}$ (resp. $w_2$). Let $\tilde \psi_i (x,t) = \psi_i (x, T t/2\pi)$. We compute
\[
\begin{array}{rlllll}
\displaystyle \partial_r \phi_1 (1) \, \displaystyle \int_{0}^{2\pi} ( H_T(w_1) \, w_2 - w_1 \, H_T(w_2) )    \, dt & = & \displaystyle \partial_r \phi_1 (1) \, \displaystyle \int_{0}^{2\pi}( \partial_r \tilde \psi_1 \, w_2 - \partial_r \tilde \psi_2 \, w_1)  \, dt\\[3mm]
& = & \displaystyle \int_{0}^{2\pi}( \tilde \psi_1 \, \partial_r \tilde \psi_2 - \tilde \psi_2  \, \partial_r  \tilde \psi_1  )  \, dt\\[3mm]
& = & \displaystyle \frac{1}{\mbox{Vol}_{g}(S^{n-1})}\int_{C_1^{2\pi}}( \tilde \psi_1  \,  \Delta_{g} \tilde \psi_2 - \tilde \psi_2 \Delta_{g} \, \tilde \psi_1 )  \, \mbox{dvol}_{ g}\\[3mm]
& = & 0.
\end{array}
\]

To prove the other statements, we define for all $v \in \mathcal \mathcal{C}^{2,\alpha}_{\textnormal{even},0} (\mathbb{R}/2\pi \mathbb{Z})$ written as in (\ref{fourier}), $\Psi$ to be the continuous solution of
\begin{equation}\label{c000}
\left\{
\begin{array}{rcll}
    \displaystyle \Delta_{g} \Psi + \lambda_{1}\,\Psi & = & \displaystyle  {\partial_r \phi_{1}} \,\sum_{j \geq 1} a_j\, \frac{1}{S_k(r)^{2}} \, \cos\left(\frac{2\pi j t}{T}\right)\, \left[ n-1 - \left(\frac{2 \pi j}{T}\right)^2 S_k(r)^2 \right]  & \textnormal{in} \qquad C_1^T \\[3mm]
    \Psi & = & 0  & \textnormal{on} \qquad \partial C_1^T \, .
\end{array}
\right.
\end{equation}
Observe that $\partial_r \phi_1$ vanishes at first order at $r=0$ and hence the right hand side is smaller than a constant times $r^{-1}$ near the origin. Standard elliptic estimates then imply that the solution $\Psi$ is at least continuous near the origin (the right side of (\ref{c000}) belongs to the space $L^p(C_1^T)$ for each $p < n$, then the solution $\Psi$ belongs to the Sobolev space $W^{2,p}(C_1^T)$ for each $p < n$, and by the Sobolev embedding theorem for a compact domain $\Omega$ we have $W^{2,p}(\Omega) \subseteq C^{0,\alpha}(\Omega)$ for $p \geq \frac{n}{2-\alpha}$). A straightforward computation using the result of Lemma~\ref{le:3.1} and writing $\Psi(x,t) = \psi(x,t) + \partial_r \phi_1(x) \, v(2 \pi t/T)$, shows that
\begin{equation}
\tilde H_T(v) : = \partial_r \Psi |_{\partial C_1^T} .
\label{Hache0}
\end{equation}

\medskip

By this alternative definition, it is clear that $H_T$ preserves the eigenspaces $V_j$ and in particular, $H_T$ maps into the space of functions whose mean is zero. \hfill $\Box$

\medskip

By the previous proposition
\begin{equation}
\label{hfourier}
\tilde H_T(v) = \sum_{j \geq 1} \sigma_j(T)\, a_j\, \cos \left(\frac{2 \pi j t}{T}\right),
\end{equation}
where $\sigma_j(T)$ are the eigenvalues of $H_T$ with respect to the eigenfunctions $\cos(jt)$. From (\ref{h11}), (\ref{hfourier}) and (\ref{c00}) we deduce that
\[
\psi = \sum_{j \geq 1} c_j(r) \, a_j\, \cos \left(\frac{2 \pi j t}{T}\right),
\]
where $c_j$ is the continuous solution on $[0,1]$ of
\begin{equation}\label{cc}
\displaystyle \left( \partial_r^2 + (n-1)\frac{C_k(r)}{S_k(r)}
 \, \partial_r  + \lambda_{1}\right)\, c_j -\left(\frac{2 \pi j}{T}\right)^2\, c_{j}= 0,
\end{equation}
with $c_j(1) = -\partial_r \phi_1(1)$. Then
\begin{equation}
\label{sigma2} \sigma_j (T) = \partial_r c_j (1) + \partial_r^2
\phi_1(1).
\end{equation}

\medskip

Our next task is to find the kernel of the operator $H_T$. For this it is enough to study the eigenvalues $\sigma_j$. We remark that if we set
\[
\frac{j}{T} = \frac{1}{D} \,,
\]
for $T >0$, from (\ref{cc}) we obtain that
\[
\sigma_j(T) = \sigma_1(D).
\]
Then, in order to study the kernel of the linearized operator, it suffices to consider only the first eigenvalue $\sigma_1$. For this aim we will use Legendre and Ferrers functions.

\medskip

To simplify the notation, in the sequel we will drop the lower
 index~$_1$, and we set $\sigma_1 = \sigma$.

\section{Recollection on Legendre and Ferrers functions}

In what follows we shall use several properties of associated
Legendre and Ferrers functions. For the convenience of the reader,
we recall their definitions and some properties. This section can be
skipped by the reader who is familiar with these functions.
 For more details we refer to \cite{D, L, O}.

\subsection{Legendre functions}
The (general) Legendre equation in the variable $z~\in~\C$ (see
\cite{O}, 5.12) is
\begin{equation}
\label{legendre.equation}
 (1-z^2)\, \frac{d^2w}{dz^2}-2
z\, \frac{dw}{dz}+ \left[\nu(\nu+1)-\frac{\mu^2}{1-z^2}\right]\, w=0
\end{equation}
where $\mu,\nu$ are complex parameters. To solve this equation one
considers special solutions to the hypergeometric equation:
$$
z(1-z) \frac{d^2u}{dz^2}+ \{c-(a+b+1)z\}\frac{du}{dz}-abu=0,
$$
where $a,b,c \in \C.$ The solutions to this equation can be found by the power series
method. If we consider a series centered at $z=0$ we find a series
which is convergent for $|z|<1$ and whose sum is known as
hypergeometric function:
$$F(a,b;c;z)=\sum_{s=0}^\infty \frac{(a)_s(b)_s}{(c)_s}
\frac{z^s}{s!},$$
where $c>0$ (see \cite{O}, 9.02, p.159).
Let $\Gamma$ be the Gamma function and let $(\cdot)_s$ denote the Pochammer
 symbol
\[
(q)_n = \left\{
\begin{array}{ll}
1 & \mbox{if } n=0\\
q\, (q+1)\, (q+2) \cdots (q+n-1) &  \mbox{if } n\geq 1
\end{array}
\right. .
\]
Tthe Olver hypergeometric function ${\bf F}$ (see \cite{O},
9.03, p.159) is defined by
\[
{\bf F}(a,b;c;z)=\frac{F(a,b;c;z)}{\Gamma(c)}=
\sum_{s=0}^\infty \frac{(a)_s(b)_s}{\Gamma(c+s)}
\frac{z^s}{s!}
\]
for $|z|<1$ and extended to $|z|\geq 1$ by analytic continuation.
Such a function presents the advantage of being defined for all
values of $c.$
Using the Olver hypergeometric function we can construct a first
solution of \eqref{legendre.equation}:
\begin{equation}
\begin{array}{lcl}
\displaystyle \PP_{\nu}^{-\mu}(z) & = & \displaystyle \left(
\frac{z-1}{z+1}\right)^{\mu/2}\, {\bf F}\left(\nu+1\,,\, -\nu\,;\,
\mu+1\,;\,\frac{1- z}{2}\right).
\end{array}
\end{equation}
A second solution can be built from the first one by using the fact
that also
$$(-z)^a\, {\bf F}\left(a,1+a-c,1+a-b,\frac1z\right)$$
is a solution to the hypergeometric equation and replacing
$a=\nu+1,$ $b=-\nu,$ $c=\mu+1$ and $z$ by $\frac{1- z}{2}.$
 We get (after multiplication by $2^\nu\,\Gamma(\nu+1)$):
\begin{equation}
\begin{array}{lcl}
\displaystyle {\bf Q}_{\nu}^{\mu}(z) & = & \displaystyle 2^\nu\,
\Gamma(\nu+1)
 \frac{(z-1)^{\mu/2-\nu-1}}{(z+1)^{\mu/2}}\,
{\bf F}\left(\nu+1\,,\, \nu-\mu+1\,;\,2\nu+2\,;\,\frac
2{1-z}\right).
\end{array}
\end{equation}
Because the  Legendre equation is unchanged by
replacing $\mu$ by $-\mu$ or $\nu$ by $-\nu-1,$ functions
$$
\PP_{\nu}^{\pm \mu}(z),\PP_{-\nu-1}^{\pm \mu}(z),
 {\bf Q}_{\nu}^{\pm \mu}(z),{\bf Q}_{-\nu-1}^{\pm \mu}(z)\,,
$$
are all solutions, but only the following four of them are distinct:
$$
\PP_{\nu}^{\pm \mu}(z), {\bf Q}_{\nu}^{\mu}(z),
{\bf Q}_{-\nu-1}^{\mu}(z) \,.
$$
Moreover only two
of them are linearly independent, as one can see by the two following connection formulas:
\begin{equation}
\label{legendre.connection}
\begin{array}{lcl}
\displaystyle \frac{2 \sin(\mu \pi)}{\pi}\, {\bf Q}_{\nu}^{\mu}(z) &
= & \displaystyle\frac{\PP_{\nu}^{ \mu}(z)}{\Gamma(\nu+\mu+1)}-
\frac{\PP_{\nu}^{- \mu}(z)}{\Gamma(\nu-\mu+1)}\\[3mm]
\displaystyle \cos(\nu \pi)\, \PP_{\nu}^{ -\mu}(z) & = &
\displaystyle \frac{{\bf Q}_{-\nu-1}^{\mu}(z)}
{\Gamma(\nu+\mu+1)}-\frac{{\bf Q}_{\nu}^{\mu}(z)} {\Gamma(\mu-\nu)}.
\end{array}
\end{equation}
The functions $\PP_{\nu}^{ \pm \mu}(z)$ are called  associated
Legendre functions of first kind. The functions ${\bf Q}_{\nu}^{ \pm
\mu}(z)$ are called associated Legendre functions of second kind\footnote{For more clarity we denote the associated Legendre
functions of first kind by $\PP^{\pm \mu}_{\nu}(x).$ We do not adopt
the standard notation $P^{\pm \mu}_{\nu}(x)$ which is very similar
to $\Fer^{\pm \mu}_{\nu}(x),$ that denotes the associated Ferrers
function of first kind.}. Such functions exist for all
values of $\nu,\mu,z,$ except possibly the singular points $z=\pm 1$
and $\infty.$ As functions of $z$ they are many valued with branch
points at $z=\pm 1$ and $\infty.$ The principal branches of both
solutions are obtained by introducing a cut along the real axis from
$z=-\infty$ to $z=+1,$ and assigning the principal value to each
function.

\subsection{Ferrers functions}

Suppose that $\PP_{\nu}^{ -\mu}(z)$ and ${\bf Q}_{\nu}^{\mu}(z)$
are real valued on the real interval  $[1,+\infty)$ (it is the case when $\nu,\mu \in \R$).
On the cut from $-\infty$ to $1$ there are two possible values
for each function, depending whether the cut is approached from
 the upper or lower side. Replacing $z$ by $x,$ we denote these
 values by $$\PP_{\nu}^{ -\mu}(x+i0),\, \PP_{\nu}^{ -\mu}(x-i0),\,
 {\bf Q}_{\nu}^{\mu}(x +i 0),\, {\bf Q}_{\nu}^{\mu}(x - i 0).$$
For $|x| < 1,$
 it is possible to define four real
valued functions if $\nu$ and $\mu$ are real.
They are known as associated Ferrers functions. Two of such functions are defined as follows under the assumption $-(\nu +\mu) \notin \mathbb{N}^*$ (here $\mathbb{N}^* = \{1,2,3, ...\}$):
\begin{equation}
\begin{array}{lcl}
\Fer_\nu^\mu(x) & = & \displaystyle e^{i\mu
\pi/2}\, \PP_{\nu}^{\mu}(x+i0)=
e^{-i\mu \pi/2}\,  \PP_{\nu}^{ \mu}(x-i0)\\[3mm]
\QFer_\nu^\mu(x) & = & \displaystyle \frac12 \, \Gamma(\nu + \mu +1)
\left[ e^{-i\mu \pi/2}\, {\bf Q}_{\nu}^{\mu}(x+i0)+ e^{i\mu \pi/2}\, {\bf
Q}_{\nu}^{ \mu}(x-i0) \right].
\end{array}
\end{equation}
The two other associated Ferrers functions are $\Fer_\nu^{-\mu}(x)$ and
$\QFer_\nu^{-\mu}(x).$
It is possible to show that
\[
\Fer_\nu^\mu(x)= \left( \frac{1+x}{1-x} \right)^{\mu/2} {\bf
F}\left(\nu+1\,,\, -\nu\,;\, 1-\mu\, ;\, \frac{1-x}2\right).
\]
Such a formula allows to extend the definition of $\Fer^\mu_\nu(x)$
to complex values of $\nu,\mu$ and $x:$ cuts are introduced along
the real intervals $(-\infty,-1]$ and $[1,+\infty)$. The expression
for other Ferrers functions can be derived using the connection
formulas:
\begin{equation}
\label{ferrers.connection}
\begin{array}{lcl}
\Fer^{\mu}_{\nu} & = & \displaystyle \frac {\Gamma(\nu+\mu+1)}{\Gamma(\nu-\mu+1)}
\left[\cos(\mu \pi)\, \Fer^{-\mu}_{\nu}+ \frac{2 \sin(\mu
\pi)}{\pi}\, \QFer_{\nu}^{-\mu}\right] \\[5mm]
\QFer^{\mu}_{\nu} & = &\displaystyle  \frac
{\Gamma(\nu+\mu+1)}{\Gamma(\nu-\mu+1)} \left[\cos(\mu \pi)\,
\QFer^{-\mu}_{\nu}- \frac{\pi \sin(\mu \pi)}{2}\,
\Fer_{\nu}^{-\mu}\right].
\end{array}
\end{equation}
In particular the formula we get for $\QFer^{\mu}_{\nu}$
is used to extend $\QFer^\mu_\nu(x)$
to complex values of $\nu,\mu$ and $x$ in the same way
as for $\Fer^\mu_\nu(x)$.

\subsection{Asymptotics}
We recall now some asymptotics about Legendre and Ferrers functions that we will need through the paper.

 \begin{lemma}(see \cite{D}, Section 14.8)
\label{asymptotics}
The associated Legendre  functions
$\PP^\mu_\nu(x),$ ${\bf Q}^\mu_\nu(x)$  defined on
$(1,+\infty)$ have the following asymptotic behaviours for $x \to
1^+$:
\begin{eqnarray}
\label{P.unbounded.half.integer}{ \PP^\mu_\nu}(x) &\sim& \frac{1}{\Gamma(1-\mu)}
\left( \frac{2}{x-1} \right)^{\frac{\mu}2} \quad {\hbox{if }} \mu
\notin \mathbb{N}^*\\[5mm]
\label{P.bounded.integer}
{ \PP^\mu_\nu}(x) & \sim &
\frac{\Gamma(\nu+\mu+1)}{\Gamma(\nu-\mu+1)\mu !} \left(
\frac{x-1}{2} \right)^{\frac{\mu}2} \quad {\hbox{if }} \mu \in \mathbb{N}^*, -(\nu \pm \mu) \notin \mathbb{N}^* \\[5mm]
{ {\bf Q}^\mu_\nu}(x) & \sim &
\frac{\Gamma(\mu)}{2\Gamma(\nu+\mu+1)} \left(
\frac{2}{x-1} \right)^{\frac{\mu}2} \quad {\hbox{if }} Re(\mu)>0,
-(\nu+\mu)  \notin \mathbb{N}^*
\label{Q.unbounded}
\end{eqnarray}
Associated Ferrers functions $ \Fer^\mu_\nu(x)$, $\QFer^\mu_\nu(x) $ have the
following asymptotic behaviour for $x \to 1^-$:
\begin{eqnarray}
\label{Fer.unbounded.integer}
\Fer^\mu_\nu(x) &\sim& \frac{1}{\Gamma(1-\mu)}
\left( \frac{2}{1-x} \right)^{\frac{\mu}2} \, , \,  \mu
 \notin \mathbb{N}^*\\[5mm]
\label{Fer.bounded.half.integer}
\Fer^\mu_\nu(x) & \sim &
\frac{\Gamma(\nu+\mu+1)\,  (-1)^\mu}{\Gamma(\nu-\mu+1)\, \mu !} \left(
\frac{1-x}{2} \right)^{\frac{\mu}2} \, , \, \mu  \in \mathbb{N}^* , \nu \neq \mu-1,\mu-2, ..., -\mu\\[5mm]
\label{QFer.integer}
\QFer_{\nu}^{\mu}(x) & \sim & \frac{1}{2} \cos(\pi \mu)
{\Gamma(\mu)} \left(\frac{2}{1-x} \right)^{\mu/2},
\mu \notin \left(\mathbb{N}^*-\frac12\right)\\[5mm]
\label{QFer.half.integer}
\QFer_{\nu}^{\mu}(x) & \sim &
\frac{\pi\, \Gamma(\nu+\mu+1)\, (-1)^{\mu+\frac12} }{2\, \Gamma(\mu+1)\Gamma(\nu-\mu+1)}
\left(\frac{1-x}{2} \right)^{\frac\mu2}\, , \, \mu  \in \left(\mathbb{N}^*-\frac12\right),  -(\nu\pm\mu)  \notin \mathbb{N}^*
\end{eqnarray}
where $\displaystyle \mathbb{N}^*-\frac12 = \left\{\frac12, \frac32, \frac52, ...\right\}$.
\end{lemma}

\section{Finding a formula for $\sigma(T)$ via Legendre and Ferrers functions}
\label{different.signs}
We are going now to study the first eigenvalue $\sigma_1(T) = \sigma(T)$ of the linearized operator $H_T$. For this we need a formula of $\sigma(T)$. Recall that
\begin{equation} \label{e:cphi}
\sigma(T) \,=\, c'(1) + \phi''(1)\, ,
\end{equation}
where $\phi(r)$ is the bounded solution of the ordinary differential
equation
\begin{equation}
\label{equation.fi} u''(r) + (n-1)\, \frac{C_k(r)}{S_k(r)}\,
u'(r)+\lambda_1\,  u(r)=0
\end{equation}
such that $\phi(1)=0$ and $\phi(r)>0$ on $[0,1)$, and normalized by
\eqref{norm1}, and $c(r)$ is the continuous solution on $[0,1]$ of
the ordinary differential equation
\begin{equation}
\label{equation.c} u''(r) + (n-1)\, \frac{C_k(r)}{S_k(r)}\, u'(r)+
\left[\lambda_1 - \left(\frac{2\pi}{T}\right)^2\right] \,  u(r)=0
\end{equation}
such that $c(1)=-\phi'(1)$. We observe that $\phi'(1) \neq 0$
otherwise $\phi(r)\equiv 0.$ Indeed the solution of
\eqref{equation.fi} satisfying also $\phi(1)=\phi'(1)=0$ is the
function identically equal to zero.

\medskip
The general solution of \eqref{equation.fi} can be found as follows.
The function
\[
p(r):=S_k(r)^{\frac n2-1} \, u(r)
\]
satisfies:
$$p''(r) + \frac{C_k(r)}{S_k(r)}\, p'(r)+
\left\{\lambda_1+k\left(\frac n2-1\right)+ \left[\left(\frac
n2-1\right)\frac{C_k(r)}{S_k(r)}\right]^2  \right\}\, p(r)=0. $$ By
the change of variable $x=x(r)=C_k(r),$ we get that the function
\[
w(x)=p(r(x))
\]
satisfies \eqref{legendre.equation} after replacing
$z$ by the real variable $x$  and setting
$$ \mu=\frac{n-2}2\,, \quad
\nu=-\frac12 + \sqrt{ \frac{(n-1)^2}4+\frac{\lambda_1}{k}} \, .
$$
When $\frac{(n-1)^2}4+\frac{\lambda_1}{k}<0$ then we will always
consider the square root having positive imaginary part. In other
terms  $Im(\nu)>0.$ The general solution to
\eqref{legendre.equation} can be expressed as linear combination of
$\PP^\mu_{\nu}(x),{\bf Q}^\mu_{\nu}(x) $ if $k<0$ and of
$\Fer^\mu_{\nu}(x),\QFer^\mu_{\nu}(x) $ if $k>0.$ Consequently the
general solution to \eqref{equation.fi} is:
\begin{equation}
u(r)=\left\{
\begin{array}{c}
a\, (S_k(r))^{1-\frac n2} \, \Fer^\mu_{\nu}(C_k(r))+
b\, (S_k(r))^{1-\frac n2} \, \QFer^\mu_{\nu}(C_k(r)) \qquad \hbox{ if }k>0\\[3mm]
a\, (S_k(r))^{1-\frac n2} \, \PP^\mu_{\nu}(C_k(r))+ b\,
(S_k(r))^{1-\frac n2} \, {\bf Q}^\mu_{\nu}(C_k(r)) \qquad  \hbox{ if
}k<0. \label{solution.4.cases}
\end{array}
\right.
\end{equation}
Lemma \ref{asymptotics} says that such functions
 are, in some cases, unbounded on $[0,1].$ They can diverge as $r$
  tends to $0$, as specified below.
${\bf Q}^\mu_{\nu}(C_k(r))$ is unbounded for $Re(\mu)>0$
and $\mu+\nu \neq -1,-2,-3, \ldots$.
$\PP^\mu_{\nu}(C_k(r))$ is unbounded if $\mu$ is half-integer (that
is $n$ is odd).
$\QFer^\mu_{\nu}(C_k(r))$ is unbounded if  $\mu$ is integer (that
is $n$ is even). $\Fer^\mu_{\nu}(C_k(r))$ is unbounded if $\mu$ is
half-integer (that is $n$ is odd).
Furthermore, the function $\QFer^\mu_{\nu}(C_k(r))$ is bounded if $\mu$
is half-integer, but it is a complex valued function.

\medskip

If $\mu$ is half-integer, then a bounded real valued solution
to equation
 \eqref{equation.fi} is  $\PP^{-\mu}_{\nu}(x)$
  if $k<0,$ and $\Fer^{-\mu}_{\nu}(x)$ if $k>0$
  (see \eqref{P.unbounded.half.integer},
  \eqref{Fer.bounded.half.integer} with $\mu$
  replaced by $-\mu$).
Formulas \eqref{legendre.connection} and \eqref{ferrers.connection}
 show that the function
 $\PP^{-\mu}_{\nu}(x)$ is a linear combination of
 $\PP^{\mu}_{\nu}(x),{\bf Q}^\mu_{\nu}(x),$ and $\Fer^{-\mu}_{\nu}(x)$
 is a linear combination
of $\Fer^\mu_{\nu}(x),$ $ \QFer^\mu_{\nu}(x).$
Consequently:
$$\phi(r)=\left\{
\begin{array}{l}
s\, (S_k(r))^{1-\frac n2} \, \Fer^\mu_\nu(C_k(r)), \hbox{ if }k>0,
\mu\hbox{ integer }\\
s\, (S_k(r))^{1-\frac n2} \, \Fer^{-\mu}_\nu(C_k(r)), \hbox{ if }k>0,
\mu\hbox{ half-integer }\\
s\, (S_k(r))^{1-\frac n2} \, \PP^\mu_\nu(C_k(r)), \hbox{ if }k<0,
\mu\hbox{ integer }\\
s\, (S_k(r))^{1-\frac n2} \, \PP^{-\mu}_\nu(C_k(r)), \hbox{ if }k<0,
 \mu\hbox{
half-integer }
\end{array}
\right.
$$
where $s$ is a constant chosen in order to ensure the conditions
$\phi(r)>0$ for $r \in [0,1)$ and \eqref{norm1}. The value of
eigenvalue $\lambda_1$ which appears in $\nu$ is  the smallest
positive real number so that $\phi(1) = 0.$

\medskip

In order to find the function $c(r)$ we set
$$\nu^*=-\frac12 +
\sqrt{\frac{(n-1)^2}4 + \frac{\lambda_1-\frac{4 \pi ^2}{T^2}}k}
\,.$$ When $\frac{(n-1)^2}4+ \frac{\lambda_1-\frac{4 \pi
^2}{T^2}}k<0$ then we will always suppose that the imaginary part of
$\nu^*$ is positive. By the same reasoning we did for $\phi$, we
find that the solution of (\ref{equation.c}) is given by
\begin{equation}
c(r)=\left\{
\begin{array}{l}
A\, (S_k(r))^{1-\frac n2} \, \Fer^\mu_{\nu^*}(C_k(r)),
 \hbox{ if }k>0,\mu \hbox{ integer}\\
 A\, (S_k(r))^{1-\frac n2} \, \Fer^{-\mu}_{\nu^*}(C_k(r)),
 \hbox{ if }k>0,\mu \hbox{ half-integer}\\
A\, (S_k(r))^{1-\frac n2}\,  \PP^\mu_{\nu^*}(C_k(r)), \hbox{ if }k<0,
\mu \hbox{ integer}\\
A\, (S_k(r))^{1-\frac n2}\, \PP^{-\mu}_{\nu^*}(C_k(r)), \hbox{ if }k<0, \mu
\hbox{ half-integer}\label{solution.ferrers}
\end{array}
\right.
\end{equation}
where $A$ is a constant that can be determined using the boundary
condition $c(1)=-\phi'(1).$

\medskip

In the next two sections we will study $\sigma(T)$. An essential ingredient will be the following:

  \begin{proposition} \label{monotonicity.zeros} The following facts hold:
\begin{enumerate}
\item Let $r_0>0$ be the $n$-th zero of the associated Legendre function
  $\PP^{\mu}_{-\frac12+i \tau}(C_k(r)).$
 If $\tau \in \R^+,$ then $r_0$ is a decreasing function of $\tau.$
\item Let $r_0 \in (0,\pi)$
be the $n$-th zero of the associated Ferrers function $\Fer^{\mu}_{-\frac12 +i \tau}(C_k(r)).$
 If $\tau \in \R^+,$ then $r_0$ is a decreasing function of $\tau.$
 \end{enumerate}
 \end{proposition}

\begin{proof}
We follow the proof of Theorem 7.6.4 in \cite{O}. Suppose that
$z_0=\cosh(r_0)$ and $\nu=-1/2+i \tau.$  If we differentiate
$\PP^{\mu}_{-1/2+i \tau}
(z_0)=0,$ we get
\begin{equation}
\label{derivative.tau}
(\PP^{\mu}_{-\frac12 +i \tau})' (z_0)\frac{dz_0}{d\tau}+
\frac{\partial \PP^{\mu}_{-\frac12+i \tau}}{\partial \tau}(z_0) =0.
\end{equation}
The differential equation satisfied by the function
$\PP^{\mu}_{\nu}$ is
$$[(1-x^2)(\PP^{\mu}_{\nu})']'+ \left(\nu(\nu+1)-
\frac{\mu}{1-x^2} \right) \PP^{\mu}_{\nu}=0.$$
We multiply it by $\PP^{\mu}_{\eta},$ with $\eta \neq \nu,$
and we subtract from the
expression we get this way, the differential equation
satisfied by $\PP^{\mu}_{\eta}$ multiplied by $P^{\mu}_{\nu}.$
We get:
$$ [(1-x^2)((\PP^{\mu}_{\nu})'\PP^{\mu}_{\eta}-
(\PP^{\mu}_{\eta})'\PP^{\mu}_{\nu})]'+(\nu(\nu+1)-\eta(\eta+1))
\PP^{\mu}_{\nu}\PP^{\mu}_{\eta}=0.$$
If  $\eta=-\frac12 +i \rho,$ then
$\nu(\nu+1)-\eta(\eta+1)=\rho^2-\tau^2.$
In conclusion, if $\rho \neq \tau,$
$$\int \PP^{\mu}_{\nu}\PP^{\mu}_{\eta} dx=
\frac{(x^2-1)((\PP^{\mu}_{\nu})'\PP^{\mu}_{\eta}-
(\PP^{\mu}_{\eta})'\PP^{\mu}_{\nu})}{\rho^2-\tau^2}.$$
If we let $\rho$ tend to $\tau,$ then using the l'H\^opital rule,
we get:
 $$\int (\PP^{\mu}_{\nu})^2 dx=
\frac{(x^2-1)}{2\tau}
\left((\PP^{\mu}_{\nu})'\frac{\partial(\PP^{\mu}_{\nu})}{\partial \tau}-\frac{\partial(\PP^{\mu}_{\nu})}{\partial \tau}'\PP^{\mu}_{\nu}
\right).$$
If we set the integration bounds equal to $1$ and $z_0$ then
\begin{equation}
\label{integral}
 \int_1^{z_0} (\PP^{\mu}_{\nu})^2 dx=
\frac{(z_0^2-1)}{2\tau}(\PP^{\mu}_{\nu})'(z_0)
\frac{\partial(\PP^{\mu}_{\nu})}{\partial \tau}(z_0).
\end{equation}
In other terms:
$$\frac{\partial(\PP^{\mu}_{\nu})}{\partial \tau}(z_0)=
\frac{2 \tau}{(z_0^2-1)} \frac{1}{(\PP^{\mu}_{\nu})'(z_0)}
 \int_1^{z_0} (\PP^{\mu}_{\nu})^2 dx$$
which replaced in \eqref{derivative.tau} gives:
$$\frac{dz_0}{d \tau}=-\frac{2 \tau}{(z_0^2-1)} \frac{1}{((\PP^{\mu}_{\nu})'(z_0))^2}
 \int_1^{z_0} (\PP^{\mu}_{\nu})^2 dx<0.$$
As $z_0=\cosh(r_0)$ then
$$\frac{dz_0}{d \tau}=\frac{dz_0}{dr_0}\frac{dr_0}{d \tau}
=\sinh(r_0) \frac{dr_0}{d \tau}.$$ So
$$\frac{dr_0}{d \tau}= \frac{1}{\sinh(r_0)}
\frac{dz_0}{d \tau}<0.$$

\medskip

 The proof of the monotonicity for the zeros of
 $\Fer^{\mu}_{\nu}$ is essentially the same.
 Suppose that $z_0=\cos(r_0).$
 As $z_0 \in (-1,1),$ we set the bounds of integration
 equal to $-1$ and $z_0.$
In this  case instead of  \eqref{integral}  we have:
$$ \int_{-1}^{z_0} (\Fer^{\mu}_{\nu})^2 dx=
\frac{(z_0^2-1)}{2\tau}(\Fer^{\mu}_{\nu})'(z_0)
\frac{\partial(\Fer^{\mu}_{\nu})}{\partial \tau}(z_0).$$
Plugging it into \eqref{derivative.tau} (which is true also for
$\Fer^{\mu}_{\nu}$), we get:
$$\frac{dz_0}{d \tau}=-\frac{2 \tau}{(z_0^2-1)}
\frac{1}{((\Fer^{\mu}_{\nu})'(z_0))^2}
 \int_{-1}^{z_0} (\Fer^{\mu}_{\nu})^2 dx>0.$$
Now we consider the identity $z_0=\cos(r_0)$ then
$$\frac{dz_0}{d \tau}=\frac{dz_0}{dr_0}\frac{dr_0}{d \tau}
=-\sin(r_0) \frac{dr_0}{d \tau}.$$ As a consequence:
$$\frac{dr_0}{d \tau}= -\frac{1}{\sin(r_0)}
\frac{dz_0}{d \tau}<0.$$
This completes the proof of the proposition. \hfill $\Box$
\end{proof}

\section{Study of $\sigma(T)$}
\label{analytic}
 It is easy to see that $\sigma(T)$ is analytic. This fact comes
from the  following remark: if $K$ is an invertible operator and $I$
is the identity, then for $T > 0$ and any continuous function $v,$
the solution $u$ of
\[
\left(K - \frac{1}{T^2}\, \rho\, I\right)\, u = v
\]
is analytic with respect to $T$ for each constant $\rho$ (this follows
from the equality
\[
(I - sK)^{-1} = \sum_{n \geq 0} s^n\, K^n
\]
for each $s \in \mathbb{R}$). Then to prove that $c$ is analytic it
suffices to take
\[
K = \left( \partial_r^2  + (n-1)\frac{C_k(r)}{S_k(r)} \, \partial_r
+ \lambda_{1} \right), \qquad v = 0, \qquad \rho =
\left(2\pi\right)^2.
\]
We conclude that $c'(1)$ is analytic with respect to
$T$,  and from \eqref{sigma2} follows the analyticity of $\sigma$. The following proposition shows the behavior of $\sigma$ at $0^+$ and $+\infty$.

\begin{proposition}
\label{limit.sigma}
The function $\sigma(T)$ satisfies
\[
\lim_{T \rightarrow 0^+} \sigma (T) = + \infty \qquad and  \qquad
\lim_{T \rightarrow +\infty} \sigma (T) = - \infty.
\]
\end{proposition}

\begin{proof} We consider independently four cases, according if the dimension $n$ is odd or even and if the curvature $k$ of $\mathbb{M}^n$ is positive ($\mathbb{S}^n$) or negative ($\mathbb{H}^n$). According to remark \ref{remark_k}, we could use $k \neq 0$ instead of $k \in \{-1, 1\}$. For this reason, in the following computation we will distinguish the case $k<0$ and $k>0$ and we do not replace $k$ by its value.
\medskip

{\bf First case: $n$ even and $k$ negative.} If $n$ is even then
$\mu$ is integer. So, we are in the case $k<0$ and $\mu$ integer,
and then the derivative of $c(r)$ is
$$c'(r)= A \, \left(1-\frac n2\right)
 S_k^{-\frac n2}(r) \, C_k(r)\, \Pmn -
k\,  A\,  S_k^{2-\frac n2}(r)\,  (\PP^{\mu}_{\nu^*})'(C_k(r))\,.$$
The last summand can be expressed in terms of $\PP^{\mu}_{\nu^*}$
and $\PP^{\mu+1}_{\nu^*}$
using formula 7.12.17, page 195 \cite{L}:
\begin{equation}
\label{identity.derivative.k.negative}
 (\PP^{\mu}_{\nu^*}(x))'=\frac{1}{x^2-1}
 \left[\sqrt{(x^2-1)}\, \PP^{\mu+1}_{\nu^*}(x)+ \mu\, x\,
 \PP^{\mu}_{\nu^*}(x)\right].
\end{equation}
If $x=C_k(r)$ then $C^2_k(r)-1=-k\, S^2_k(r)$ and
 $\sqrt{C^2_k(r)-1}=\sqrt{-k}\,S_k(r).$
 As a consequence
$$  (\PP^{\mu}_{\nu^*})'(C_k(r))=
\frac1{-k\, S_k^2(r)} \left[\sqrt{-k}\,S_k(r)\,
\Pmpn + \mu\, C_k(r)\, \Pmn \right]
  $$
 and
 \[
 \begin{array}{rcl}
 c'(r) & = & \displaystyle A \left(1- \frac n2 \right) S_k^{-\frac n2}(r)\, C_k(r)\,
 \Pmn \, +\\[4mm]
  & & \qquad +\, A \, S_k^{-\frac n2}(r)
 \left[ \sqrt{-k}\, S_k(r)\,
\Pmpn + \mu\, C_k(r)\, \Pmn  \right]=\\[4mm]
& = & \displaystyle A\, S^{-\frac n2}_k(r)\left[
C_k(r)\,  \Pmn\, \left(1-\frac n2 + \mu\right) +
\sqrt{-k}\,  S_k(r) \, \Pmpn \right]\\[4mm]
& = & A\,  \sqrt{-k}\,  S^{1-\frac n2}_k(r) \, \Pmpn \, .
\end{array}
\]
If we replace $\nu^*$ by $\nu$ and $A$ by $s,$
then $c(r)$ reduces to $\phi(r).$
So the computation above shows also that
$$\phi'(r)= s\, \sqrt{-k} \, S_k^{1-\frac n2}(r)\,
\PP^{\mu+1}_{\nu} (C_k(r))\, .$$
As
$$A=-\frac{\phi'(1)}
{S_k^{1-\frac n2}(1) \PP^{\mu}_{\nu^*} (C_k(1))} = -\frac{s\, \sqrt{-k} \,
\PP^{\mu+1}_{\nu} (C_k(1))}
{\PP^{\mu}_{\nu^*} (C_k(1))} \, ,$$ then the function
$c'(r)$ is
$$c'(r)=\frac{s\,k\, \PP^{\mu+1}_{\nu} (C_k(1))}
{\PP^{\mu}_{\nu^*} (C_k(1))} \, S_k^{1-\frac n2}(r)\,  \PP^{\mu+1}_{\nu^*}
(C_k(r))\, .$$
Consequently
\[
\begin{array}{rcl}
c'(1)+\phi''(1) & = & \displaystyle 
s\, k\, S_k^{1-\frac n2} (1)\,
 \frac{\PP^{\mu+1}_\nu(C_k(1))}
{\PP^{\mu}_{\nu^*}(C_k(1))}\,
 \PP^{\mu+1}_{\nu^*} (C_k(1))+\phi''(1)\, .
 \end{array}
 \]

 \medskip

We remark that
\[
\lim_{T \to +\infty} \nu^*=\nu \, .
\]
Consequently the numerator of $c'(1)$ tends to
\[
s\, k\, S_k^{1-\frac n2} (1)\, (\PP^{\mu+1}_\nu(C_k(1)))^2
\]
when $T$ goes to $+\infty$. As $Im(\nu^*)< Im(\nu)$, then Proposition
\ref{monotonicity.zeros} ensures that the first positive zero of
$\PP^{\mu}_{\nu^*}(C_k(r))$ is bigger than $1.$ Furthermore
$\PP^{\mu}_{\nu^*}(C_k(1))>0$ if $\PP^{\mu}_{\nu}(C_k(r))>0$ for $r
\in [0,1)$ or $\PP^{\mu}_{\nu^*}(C_k(1))<0$ if
$\PP^{\mu}_{\nu}(C_k(r))<0$ for $r \in [0,1)$. By definition, $s$
has the same sign as $\PP^{\mu}_{\nu}(C_k(r))$ on $r \in [0,1).$
 Then, if $s>0,$
$$\lim_{\nu^* \to \nu} \PP^{\mu}_{\nu^*}(C_k(1))= 0^+(=
\PP^{\mu}_{\nu}(C_k(1))),$$ i.e.
$$\lim_{\nu^* \to \nu} \frac 1 {\PP^{\mu}_{\nu^*}(C_k(1))}
=+\infty \,.$$  Similarly
$$\lim_{\nu^* \to \nu} \frac 1 {\PP^{\mu}_{\nu^*}(C_k(1))}
=-\infty \,, $$ if $s<0.$
 As the second summand $\phi''(1)$ does not
depend on $T$ and it is bounded, we conclude that
$$\lim_{T \to +\infty} \sigma(T) = \lim_{T \to +\infty} [c'(1)+\phi''(1)] =
    -\infty. $$


\medskip

Now we consider the limit of $\sigma(T)$ as $T \to 0^+$.
As $k <0$ then
$$\lim_{T\to 0^+} \nu^*=\lim_{T\to 0^+}-\frac 12+
\sqrt{\frac{(n-1)^2}4 + \frac{\lambda_1-\frac{4 \pi ^2}{T^2}}k} :=
\nu_\infty =+\infty.$$ That says also that for $T$ small enough,
$\nu^*$ is real.
For $\nu^*$ big enough
$$c'(1) \sim -\phi'(1) \sqrt{-k}\,
\frac{\PP^{\mu+1}_{\nu^*}(C_k(1))} {\PP^{\mu}_{\nu^*}(C_k(1))} .$$
Formula 14.15.13 \cite{D} provides the asymptotic behaviour of
$\PP^{-\mu}_{\nu^*} $ with respect to $\nu^*$:
\begin{equation}
\label{P-mu} \PP^{-\mu}_{\nu^*}(C_k(1))\sim \frac 1{(\nu^*)^\mu}\,
\sqrt{\frac{1}{\sinh(1)}}
 \, I_\mu\left(\nu^*+\frac12\right)
 \end{equation}
where $I_\mu$ denotes the modified Bessel function of first kind (we
refer to \cite{L} for basic facts about Bessel functions). To get
the asymptotic expression for $\PP^{\mu}_{\nu^*}(C_k(1))$ we use the
following identity 
\begin{equation}
\label{P+-} \PP^{\mu}_{\nu^*}=\frac
{\Gamma(\nu^*+\mu+1)}{\Gamma(\nu^*-\mu+1)}\, \PP^{-\mu}_{\nu^*}
\end{equation}
which follows from \eqref{legendre.connection} using the fact that
$\mu$ is integer.
 Notice that $\frac
{\Gamma(\nu^*+\mu+1)}{\Gamma(\nu^*-\mu+1)} \sim (\nu^*)^t$ for
$\nu^* $ big, where $t=2\mu$ if $\mu$ is integer and $t=2\mu+1$ if
$\mu$ is not integer.
We are considering the case $\mu$ is integer, then from \eqref{P-mu}
we get
 \begin{equation}
 \label{P.mu.integer}
\PP^{\mu}_{\nu^*}(C_k(1)) \sim (\nu^*)^{\mu}\,
\sqrt{\frac{1}{\sinh(1)}}
 \, I_\mu\left(\nu^*+\frac12\right)
 \end{equation}
for $\nu^*$ big.
Observe that
\begin{equation}
\label{besseli.exp} I_\mu\left(\nu^*+\frac12\right) \sim \frac{e^{\nu^*+\frac12}}
{\sqrt{\pi(2 \nu^*+1)}}
\end{equation}
for $\nu^*$ big.
This implies that
$$\frac{\PP^{\mu+1}_{\nu^*}(C_k(1))}{\PP^{\mu}_{\nu^*}(C_k(1)) } \sim
 \nu^* \, \frac{I_{\mu+1}\left(\nu^*+\frac12\right)}{I_{\mu}\left(\nu^*+\frac12\right)}\sim
\nu^*$$
for $\nu^*$ big,
and in conclusion
$$c'(1)\sim - \phi'(1)\, \sqrt{-k}\, \nu^* $$
for $\nu^*$ big.
As $\phi'(1)<0,$ and $\phi''(1)$ is bounded and independent from
$T$, we conclude that
$$\lim_{T\to 0^+} \sigma(T) =\lim_{T\to 0^+} [c'(1)+\phi''(1)]=+\infty.$$

\medskip

{\bf Second case: $n$ odd and $k$ negative.}
If $n$ is odd, then $\mu$ is half-integer. If $k<0$ and $\mu$ is half-integer, then $c(r)$ is given by
$$c(r)=A\,S_k^{1-\frac n2}(r)\,\PP^{-\mu}_{\nu^*}(C_k(r))$$
where $A$ is the constant such that $c(1)=-\phi'(1)$. Moreover
$$\phi(r)=s \, S_k^{1-\frac n2}(r)\, \PP^{-\mu}_{\nu}(C_k(r))$$
where $s$ is a constant such that $\phi(r)>0$ for $r\in [0,1)$ and \eqref{norm1}. Moreover we have
$\phi(1)=0$.
Using \eqref{identity.derivative.k.negative} with $\mu$
replaced by $-\mu,$ we get:
 $$  (\PP^{-\mu}_{\nu^*})'(C_k(r))=
\frac1{-k\, S_k^2(r)} \left[\sqrt{-k}\, S_k(r)\,
\PP^{-\mu+1}_{\nu^*}(C_k(r)) -
\mu \, C_k(r)\,  \PP^{-\mu}_{\nu^*}(C_k(r)) \right]
  $$
 and
 \[
 \begin{array}{rcl}
 c'(r) & = & \displaystyle A \left(1- \frac n2 \right) S_k^{-\frac n2}(r) \, C_k(r)
 \PP^{-\mu}_{\nu^*} (C_k(r)) \, +\\[4mm]
 & & \displaystyle \qquad +\, A\,  S_k^{-\frac n2}(r) \left[ \sqrt{-k}\, S_k(r)\,
\PP^{-\mu+1}_{\nu^*}(C_k(r)) - \mu\,  C_k(r) \, \PP^{-\mu}_{\nu^*}(C_k(r))
\right] \\[4mm]
& = & \displaystyle A \, S^{-\frac n2}_k(r)\left[
C_k(r)\,  \PP^{-\mu}_{\nu^*}(C_k(r))\, \left(1-\frac n2 - \mu\right)+
\sqrt{-k} \, S_k(r) \, \PP^{-\mu+1}_{\nu^*}(C_k(r)) \right]\\[4mm]
& = & \displaystyle A\, [\sqrt{-k} \, S^{1-\frac n2}_k(r) \, \PP^{-\mu+1}_{\nu^*}(C_k(r))-
2 \mu \, C_k(r) \, S^{-\frac n2}_k(r) \, \PP^{-\mu}_{\nu^*}(C_k(r))]
\end{array}
\]
If we replace $\nu^*$ by $\nu$ and $A$ by $s,$
then $c(r)$ reduces to $\phi(r).$
So the computation above shows also that
$$\phi'(r)= s\, [\sqrt{-k} \,S^{1-\frac n2}_k(r)\, \PP^{-\mu+1}_{\nu}(C_k(r))-
2 \mu \, C_k(r) S^{-\frac n2}_k(r) \, \PP^{-\mu}_{\nu}(C_k(r))].$$
As a consequence $$\phi'(1)=s\, \sqrt{-k}\, S^{1-\frac n2}_k(1)\,
 \PP^{-\mu+1}_{\nu}(C_k(1))$$
because $\PP^{-\mu}_{\nu}(C_k(1))=0$.
From $c(1)=-\phi'(1),$ we get the value of the constant $A$:
$$A=- \frac{\phi'(1)}{S^{1-\frac n2}_k(1)\, \PP^{-\mu}_{\nu^*}(C_k(1))}
=-s\, \sqrt{-k}\, \frac{  \PP^{-\mu+1}_{\nu}(C_k(1))}
 { \PP^{-\mu}_{\nu^*}(C_k(1))}.$$

\medskip

 If $T \to 0^+,$ then $\nu^* \to
+\infty.$ If $\nu^*$ is big enough, then
\eqref{P-mu} gives the asymptotic behaviour for $\nu^*$ big:
$$\PP^{-\mu}_{\nu^*}(C_k(1))\sim \frac 1{(\nu^*)^{\mu}} \, \sqrt{\frac{1}{\sinh(1)}}
 \, I_\mu\left(\nu^*+\frac12\right) \, .$$
The asymptotic behaviour of $I_\mu$ is described by
\eqref{besseli.exp}.
Consequently
\[
\begin{array}{rcl}
c'(1) & = & \displaystyle A [\sqrt{-k} \, S^{1-\frac n2}_k(1)\,
\PP^{-\mu+1}_{\nu^*}(C_k(1))- 2 \mu \, C_k(1)\,  S^{-\frac n2}_k(1)
  \PP^{-\mu}_{\nu^*}(C_k(1))] \\[4mm]
& = & \displaystyle \frac{-\phi'(1)}{S^{1-\frac n2}_k(1) \, \PP^{-\mu}_{\nu^*}(C_k(1))}
\sqrt{-k} \, S^{1-\frac n2}_k(1) \, \PP^{-\mu+1}_{\nu^*}(C_k(1)) \\[4mm]
& \sim & \displaystyle -\phi'(1)\, \sqrt{-k} \, \frac{(\nu^*)^{\mu}}{(\nu^*)^{\mu-1}}
\frac{I_\mu\left(\nu^*+\frac12\right)}{I_{\mu-1}\left(\nu^*+\frac12\right)} \\[4mm]
& \sim &
-\phi'(1)\, \sqrt{-k}\, \nu^*
\end{array}
\]
As $\phi'(1)<0,$ $k<0$, and $\phi''(1)$ is bounded and independent
from $T,$ we conclude that $$\lim_{T \to 0^+} \sigma(T) = \lim_{T \to 0^+}
[c'(1)+\phi''(1)]=+\infty.$$

\medskip


It remains to study the behaviour of $\sigma(T)$ as $T \to +\infty.$
In this case $\nu^* \to\nu.$ Proposition \ref{monotonicity.zeros}
ensures that, the first positive zero of $\PP^{-\mu}_{\nu^*}(C_k(r))$
is bigger than $1.$ Consequently
$$\lim_{\nu^*\to \nu}\frac1{\PP^{\mu}_{\nu^*}(C_k(1))}=+\infty $$
if $\PP^{\mu}_{\nu}(C_k(r))>0$ on $[0,1)$ (that is $s>0$) and
$$ \lim_{\nu^*\to \nu}\frac1{\PP^{\mu}_{\nu^*}(C_k(1))}=-\infty $$
if $\PP^{\mu}_{\nu}(C_k(r))<0$ on $[0,1)$ (that is $s>0$). In other
terms such a limit has the same sign as $s.$
 Moreover the numerator of $c'(1)$ tends to
$$-s \, S^{1-\frac n2}_k(1)\, [\sqrt{-k}  \, \PP^{-\mu+1}_{\nu}(C_k(1))] ^2.$$
As the second summand
$\phi''(1)$ does not depend on $T$ and it is bounded,  then
$$\lim_{T \to +\infty} \sigma(T) = \lim_{T \to +\infty} [c'(1)+\phi''(1)] =  -\infty \,.$$

\medskip

{\bf Third case: $n$ even and $k$ positive.} If $n$ is even then $\mu$ is integer.
If $k>0$ and $\mu$ is integer then the function $c(r)$
is given by the first line of \eqref{solution.ferrers}.
As a consequence
$$c'(r)= A \left(1-\frac n2\right)
 S_k^{-\frac n2}(r) \, C_k(r)\, \Fmn -
k \, A\,  S_k^{2-\frac n2}(r)\,  ({\rm P}^{\mu}_{\nu^*})'(C_k(r)).$$
The derivative
$({\rm P}^{\mu}_{\nu^*})'(x)$ is expressed in terms of
${\rm P}^{\mu+1}_{\nu^*}(x)$ and ${\rm P}^{\mu}_{\nu^*}(x)$
using
\begin{equation}
\label{identity.derivative.k.positive}
(\Fer^{\mu}_{\nu}(x))'=\frac{1}{x^2-1 }
\left(\sqrt{1-x^2}\, \Fer^{\mu+1}_{\nu}(x) +x\, \mu\,   \Fer^{\mu}_{\nu}(x)
\right).
\end{equation}
Replacing $x$ by $C_k(r)$ we get:
$$(\Fer^{\mu}_{\nu})'(C_k(r))=
\frac1{-k \, S_k^2(r)} \left[\sqrt{k}\, S_k(r)\,  \Fer^{\mu+1}_{\nu}(C_k(r))
+ C_k(r) \, \mu\,   \Fer^{\mu}_{\nu}(C_k(r)) \right] \, .
$$
from which it follows:
\[
\begin{array}{rcl}
c'(r) & = & A \left(1- \frac n2 \right) S_k^{-\frac n2}(r) \, C_k(r)\,
 \Fer^{\mu}_{\nu^*}(C_k(r)) \, + \\[4mm]
& & \qquad +A\,  S_k^{-\frac n2}(r)
 \left[ \sqrt{k}\, S_k(r)\,
\Fer^{\mu+1}_{\nu^*}(C_k(r))
 +\mu \, C_k(r)\,  \Fer^{\mu}_{\nu^*}(C_k(r))  \right]\\[4mm]
& = & A \, S^{-\frac n2}_k(r)\left[
C_k(r) \, \Fer^{\mu}_{\nu^*}(C_k(r))\,  (1-\frac n2 + \mu)-
\sqrt{k} \, S_k(r) \, \Fer^{\mu+1}_{\nu^*}(C_k(r)) \right]\\[4mm]
& = & A \, \sqrt{k}\,  S^{1-\frac n2}_k(r)\,
\Fer^{\mu+1}_{\nu^*}(C_k(r)).
\end{array}
\]
The constant $A$ is determined in order to have $c(1)=-\phi'(1).$
The function $\phi$ is defined by
$$\phi(r)=s\, S^{1-\frac n2}_k(r)\, \Fer^{\mu}_{\nu}(C_k(r)),$$
where $s$ is the constant such that $\phi(r)>0$ for $r \in [0,1).$ To
get the expression of its derivative, we replace $A$ by $s$ and
$\nu^*$ by $\nu$ in the expression of $c'(r):$
$$\phi'(r)=s\, \sqrt{k} \, S^{1-\frac n2}_k(r)\,
\Fer^{\mu+1}_{\nu}(C_k(r)).$$
The value of the constant $A$ is given by
$$A=-\frac{\phi'(1)}{S^{1-\frac n2}_k(1) \,
 \Fer^{\mu}_{\nu^*}(C_k(1))}=-
 s\, \sqrt{k} \,
 \frac{ \Fer^{\mu+1}_{\nu}(C_k(1))}{\Fer^{\mu}_{\nu^*}(C_k(1))}.$$
So $c'(1)$ is given by
$$c'(1)=-\phi'(1)\, \sqrt{k}\, \frac{ \Fer^{\mu+1}_{\nu^*}(C_k(1))}
{\Fer^{\mu}_{\nu^*}(C_k(1))} =-s \, k\,  \frac{
\Fer^{\mu+1}_{\nu}(C_k(1))}{\Fer^{\mu}_{\nu^*}(C_k(1))}\,  S^{1-\frac
n2}_k(r)\,  \Fer^{\mu+1}_{\nu^*}(C_k(1))\,. $$
In conclusion, if $k>0$ and $\mu$ is integer,
 then $\sigma(T)=c'(1)+\phi''(1)$ equals
\begin{equation}
\label{sigma.k.positive}
 -s\, k \, S^{1-\frac n2}_k(1)\,
\frac{ \Fer^{\mu+1}_{\nu}(C_k(1))}{\Fer^{\mu}_{\nu^*}(C_k(1))}\,
 \Fer^{\mu+1}_{\nu^*}(C_k(1))
+  \phi''(1).
\end{equation}

\medskip

For $T$ big enough $\nu^*$ is a real valued
increasing function of $T.$
If $T \to +\infty, $ then $\nu^* \to \nu$ (which is a real number
in this case).
If $\Fer^{\mu}_{\nu}(C_k(r)) >0$ for $r \in [0,1),$
(that is $s>0$) then, from Proposition \ref{monotonicity.zeros}, we get
$$\lim_{\nu^* \to \nu} \Fer^{\mu}_{\nu^*}(C_k(1))= 0^+
(= \Fer^{\mu}_{\nu}(C_k(1))).$$
So $$\lim_{T \to +\infty} \frac1{\Fer^{\mu}_{\nu^*}(C_k(1))}=
+\infty\, .$$
Similarly, if $\Fer^{\mu}_{\nu}(C_k(r)) <0$ for $r \in [0,1),$
(that is $s<0$), then
$$\lim_{T \to +\infty} \frac1{\Fer^{\mu}_{\nu^*}(C_k(1))}=
-\infty \, .$$
In other terms the sign of such a limit is the same as $s.$
When $T \to +\infty,$ the numerator of $c'(1)$ tends to
$$-s \, k\, S_k^{1-\frac n2}(1)\, [\Fer^{\mu+1}_\nu(C_k(1))]^2 \, .$$
Consequently, as $k>0,$
  $$\lim_{T \to +\infty} \sigma(T) =\lim_{T \to +\infty} [c'(1)+\phi''(1)] = -\infty \, .$$

  \medskip

Now we study the limit of $\sigma(T)$ as $T \to 0^+.$
If $T \to 0^+$ then  $\nu^* \to \nu_\infty:=-1/2 + i \infty.$
We set $\nu^*=-1/2 + i \tau.$
In this case we use the following asymptotic formula (exercise 13.4, page 473 \cite{O}):
$$\Fer^{-\mu}_{-\frac12 + i \tau}(C_k(1))
=\frac 1{\tau^{\mu}}\,  \sqrt{\frac {1} {\sin(1)}}\,
I_{\mu}(\tau)\, \left(1+\mathcal{O}\left(\frac1\tau\right)\right)$$ when $\tau$ goes to infinity. To get the
corresponding formula for $\Fer^{\mu}_{-\frac12 + i \tau}(C_k(1))$ we
use
 the following identity (formula 14.9.2 \cite{D}):
\begin{equation}
\label{Ferrers+-}
\Fer^{\mu}_{\nu^*}=\frac {\Gamma(\nu^*+\mu+1)}{\Gamma(\nu^*-\mu+1)}
\left[\cos(\mu \pi) \, \Fer^{-\mu}_{\nu^*}+ \frac{2 \sin(\mu
\pi)}{\pi}\, \QFer_{\nu^*}^{-\mu}\right].
\end{equation}
As $\mu$ is integer, then \eqref{Ferrers+-} reduces to:
$$\Fer^{\mu}_{\nu^*}(x)=(-1)^\mu\,
\frac{\Gamma(\nu^*+\mu+1)}{\Gamma(\nu^*-\mu+1)}\,
\Fer^{-\mu}_{\nu^*}(x)\, .$$

We need to estimate the limit as $T \to 0^+$ of
 \begin{equation}
\label{Ferrers+mu-mu}
\begin{array}{rcl}
\displaystyle \frac{\Fer^{\mu+1}_{\nu^*}(x)}{\Fer^{\mu}_{\nu^*}(x)} &
= & \displaystyle -\frac{\Gamma(\nu^*+\mu+2)}{\Gamma(\nu^*-\mu)}
\frac{\Gamma(\nu^*-\mu+1)}{\Gamma(\nu^*+\mu+1)}
\frac{\Fer^{-\mu-1}_{\nu^*}(x)}{\Fer^{-\mu}_{\nu^*}(x)}\\[4mm]
& = & \displaystyle -(\nu^*+\mu+1)(\nu^*-\mu)
\frac{\Fer^{-\mu-1}_{\nu^*}(x)}{\Fer^{-\mu}_{\nu^*}(x)}.
\end{array}
\end{equation}
Observe that
$$(\nu^*+\mu+1)(\nu^*-\mu)=(\nu^*)^2+\nu^*-\mu-\mu^2=
\left( -\frac 12 + i \tau \right)^2 -\frac 12 + i \tau -\mu-\mu^2=$$
$$ -\frac 14 -\tau^2- \mu - \mu^2<0.$$
This implies that
\begin{equation}
\label{Fer.on.Fer}
\frac{\Fer^{\mu+1}_{\nu^*}(C_k(1))}{\Fer^{\mu}_{\nu^*}(C_k(1)) }
\sim \tau^2 \frac{\tau^\mu}{\tau^{\mu+1}}
\frac{I_{\mu+1}(\tau)}{I_{\mu}(\tau)}\sim \tau,
\end{equation}
for $\tau$ big, since
$$I_{\mu}(\tau) \sim \frac{e^{\tau}}{\sqrt{2 \pi \tau}}$$
(formula 5.16.5 \cite{L}).
In conclusion for $\tau$ big,
$$ c'(1) \sim -\phi'(1) \, \sqrt{k}\,
\frac{\Fer^{\mu+1}_{\nu^*}(C_k(1))}{\Fer^{\mu}_{\nu^*}(C_k(1)) }\sim
-\phi'(1)\,  \sqrt{k}\, \tau.$$
As $\phi'(1)<0$ and $\phi''(1)$ is bounded and does not depend on
$T$, then we conclude that
$$\lim_{T\to 0^+} \sigma(T) = \lim_{T\to 0^+} [c'(1)+\phi''(1)]= +\infty.$$

\medskip

{\bf Fourth case: $n$ odd and $k$ positive.} If $n$ is odd then
$\mu$ is half-integer. If $k>0$ and $\mu$ is half-integer, then
$c(r)$ is given by
$$c(r)=A\, S_k^{1-\frac n2}(r)\, \Fer^{-\mu}_{\nu^*}(C_k(r))$$
where $A$ is the constant such that
$c(1)=-\phi'(1)$. Moreover
$$\phi(r)=s \, S_k^{1-\frac n2}(r) \, \Fer^{-\mu}_{\nu}(C_k(r))$$
where $s$ is the constant such that $\phi(r)>0$ for $r\in (0,1)$ and \eqref{norm1}. Moreover we have
$\phi(1)=0.$
Using \eqref{identity.derivative.k.positive} with $\mu$
replaced by $-\mu,$ we get:
 $$  (\Fer^{-\mu}_{\nu^*})'(C_k(r))=
\frac1{-k\,  S_k^2(r)} \left[\sqrt{k}\, S_k(r)\,
\Fer^{-\mu+1}_{\nu^*}(C_k(r)) -\mu\,  C_k(r)\,
\Fer^{-\mu}_{\nu^*}(C_k(r)) \right]
  $$
 and
 \[
\begin{array}{rcl}
c'(r) & = & \displaystyle A \left(1- \frac n2 \right) S_k^{-\frac n2}(r)\,  C_k(r)\,
 \Fer^{-\mu}_{\nu^*} (C_k(r))\, + \\[4mm]
 & & \displaystyle \qquad +\, A\,  S_k^{-\frac n2}(r) \left[ \sqrt{k}\, S_k(r)\,
\Fer^{-\mu+1}_{\nu^*}(C_k(r)) - \mu\,  C_k(r) \, \Fer^{-\mu}_{\nu^*}(C_k(r))
\right]\\[4mm]
& = & \displaystyle A \, S^{-\frac n2}_k(r)\left[
C_k(r)\,  \Fer^{-\mu}_{\nu^*}(C_k(r)) (1-\frac n2 - \mu)+
\sqrt{k} \, S_k(r)\, \Fer^{-\mu+1}_{\nu^*}(C_k(r)) \right] \\[4mm]
& = & \displaystyle A\, [\sqrt{k} \, S^{1-\frac n2}_k(r)\,
\Fer^{-\mu+1}_{\nu^*}(C_k(r))- 2 \mu \, C_k(r) \, S^{-\frac n2}_k(r)
\,  \Fer^{-\mu}_{\nu^*}(C_k(r))].
\end{array}
\]
If we replace $\nu^*$ by $\nu$ and $A$ by $s,$
then $c(r)$ reduces to $\phi(r).$
So the computation above shows also that
$$\phi'(r)= s [\sqrt{k}\,  S^{1-\frac n2}_k(r) \, \Fer^{-\mu+1}_{\nu}(C_k(r))-
2 \mu \, C_k(r)\,  S^{-\frac n2}_k(r)  \, \Fer^{-\mu}_{\nu}(C_k(r))]$$ from
which
$$\phi'(1)= s \, \sqrt{k} \, S^{1-\frac n2}_k(1) \, \Fer^{-\mu+1}_{\nu}(C_k(1))\, .$$
From $c(1)=-\phi'(1),$ we get the value of the constant $A$:
$$A=\frac{-\phi'(1)}{ S^{1-\frac n2}_k(r) \, \Fer^{-\mu}_{\nu^*}(C_k(r))}
=-s\,  \sqrt{k}\,  \frac{\Fer^{-\mu+1}_{\nu}(C_k(1))}
{\Fer^{-\mu}_{\nu^*}(C_k(1))}\, .$$

\medskip

If $T \to 0^+$ then $Im(\nu^*)\to \tau_\infty =+ \infty.$ For $T$ big
enough, $\nu^*$ is a real number and if $T \to +\infty,$ then $\nu^*
\to \nu \in \R.$ If $\nu^*$ is big enough, then
\eqref{P-mu} gives the asymptotic behaviour for $\nu^*$ big:
$$P^{-\mu}_{\nu^*}(C_k(1))\sim \frac 1{(\nu^*)^{\mu}} \sqrt{\frac{1}{\sinh(1)}}
 \, I_\mu((\nu^*+1/2)).$$
 The asymptotic behaviour of $I_\mu$ is described by
\eqref{besseli.exp}.
Consequently $$c'(1)=A [\sqrt{k} S^{1-\frac n2}_k(1)
\Fer^{-\mu+1}_{\nu^*}(C_k(1))- 2 \mu C_k(1) S^{-\frac n2}_k(1)
  \Fer^{-\mu}_{\nu^*}(C_k(1))] \sim$$
$$ -\phi'(1) \sqrt{k} \frac{\Fer^{-\mu+1}_{\nu^*}(C_k(1))}{\Fer^{-\mu}_{\nu^*}(C_k(1))}
\sim -\phi'(1) \sqrt{k}  \frac{(\nu^*)^{\mu}}{(\nu^*)^{\mu-1}}
\frac{I_\mu((\nu^*+1/2))}{I_{\mu-1}((\nu^*+1/2))}\sim
-\phi'(1)\sqrt{k}\nu^*.$$
As $- \phi'(1)>0$ and $\phi''(1)$ is bounded and independent from
$T,$ then we conclude that $$\lim_{T \to 0}
[c'(1)+\phi''(1)]=+\infty.$$

\medskip

It remains to study the behaviour of $\sigma(T)$ as $T \to +\infty.$
If $T \to +\infty$ then $\nu^* \to \nu$ increasing (for $T$ big
enough $\nu^*$ is real). Proposition \ref{monotonicity.zeros}
ensures that, the first positive zero of
$\Fer^{-\mu}_{\nu^*}(C_k(r))$ is bigger than $1.$ Consequently
$$\lim_{\nu^*\to \nu}\frac1{\Fer^{\mu}_{\nu^*}(C_k(1))}=+\infty$$
if $P^{\mu}_{\nu}(C_k(r))>0$ on $[0,1)$ (that is $s>0$) and
$$ \lim_{\nu^*\to \nu}\frac1{\Fer^{\mu}_{\nu^*}(C_k(1))}=-\infty$$
if $\Fer^{\mu}_{\nu}(C_k(r))<0$ on $[0,1)$ (that is $s<0$). In other
terms such a limit has the same sign as $s.$
 The numerator of $c'(1),$ tends to
$$-s \, k\,  S^{1-\frac n2}_k(1)\,  [\Fer^{-\mu+1}_{\nu}(C_k(1))] ^2 \,.$$
As the second summand
$\phi''(1)$ does not depend on $T$ and it is bounded,  then
$$\lim_{T \to +\infty} \sigma(T) = \lim_{T \to +\infty} [c'(1)+\phi''(1)] =  -\infty \,.$$
This completes the proof of the proposition. \hfill $\Box$
\end{proof}

 \section{Lyapunov-Schmidt reduction and bifurcation}

In view of the analyticity of $\sigma$ (showed in section
\ref{analytic}) and Proposition \ref{limit.sigma}, $\sigma$ has at
least a zero and the set of the zeros of $\sigma$ is finite. Let
$\{0_1, 0_2, ..., 0_p\}$ denotes the set of the zeros of $\sigma$,
and let $T_*$ be the smallest zero such that $\sigma$ changes sign
at $T_*$, say $T_* = 0_q$ (the existence of $T_*$ follows also from
the analyticity of $\sigma$ and Proposition \ref{limit.sigma}). It
is clear then the eigenspace $V_1$ (defined in Proposition \ref{H})
belongs to the kernel of $H_{T_*}$. As $\sigma_j(T) = \sigma(T/j)$
we obtain that $\sigma_j$ is analytic on $T$ and the set of the
zeros of $\sigma_j$ is $\{j\, 0_1, j\, 0_2, ..., j\, 0_p\}$. It is
clear that if $j$ is big enough then $T_* \notin \{j\, 0_1, j\, 0_2,
..., j\, 0_p\}$, and this means that $V_j$ does not belong to the
kernel of $H_{T_*}$ for almost all $j$. This implies that the kernel
of $H_{T_*}$ is of the form $V_{j_1} \oplus \cdots \oplus V_{j_l}$
with $1 = j_1 < \cdots < j_l$. Moreover if $V_{j_i} \in
\textnormal{Ker}(H_{T_*})$ and $j_i \neq 1$ then the function
$\sigma_{j_i}(T)$ does not change sign at $T_*$ by the definition of
$T_*$.
\medskip

We summarize such facts in the following proposition, where we use also the ellipticity of the linearized operator $H_T$ given by Proposition \ref{H}.
\begin{proposition}
   \label{kernel.HT}
There exists a positive real number $T_*$ such that the kernel of $H_{T_*}$ is given by $V_{j_1} \oplus \cdots \oplus V_{j_l}$, with $1 = j_1 < \cdots < j_l$. Moreover the eigenvalue associated to the eigenspace $V_1$, considered as a function on $T$, changes the sign at $T_*$, and the eigenvalues associated to the other eigenspaces $V_{j_2},..., V_{j_l}$, always considered as functions on $T$, do not change sign at $T_*$. There exists a constant $c >0$ such that
\[
\| w \|_{\mathcal \mathcal{C}^{2,\alpha}_{\textnormal{even},0}
(\mathbb{R}/2\pi \mathbb{Z})}  \leq c \, \| H_{T_*}(w)
\|_{\mathcal C^{1,\alpha}_{\textnormal{even},0}
(\mathbb{R}/2\pi \mathbb{Z})} \, ,
\]
provided $w$ is $L^2(\mathbb{R}/2\pi \mathbb{Z})$-orthogonal to $V_0 \oplus V_{j_1} \oplus \cdots \oplus V_{j_l}$, where $V_0$ is the space of constant functions.
\end{proposition}

Such proposition says us that the operator $H_{T_*}$ has finite-dimensional kernel, and that it is an isomorphism from the orthogonal to its kernel over its image (see also Proposition \ref{H} and its proof). We are going to use now these two properties.

\medskip

Consider the space $C^{2,\alpha}_{\textnormal{even},0} (\mathbb{R}/2\pi
 \mathbb{Z}) \times (0,+\infty)$. Clearly the curve
\[
\Xi = \{(v,T) \quad : \quad v \equiv 0 \}
\]
in $\mathcal \mathcal{C}^{2,\alpha}_{\textnormal{even},0} (\mathbb{R}/2\pi
 \mathbb{Z}) \times (0,+\infty)$ belongs to the zero level set of the operator $F$, i.e. its points solve the equation
 \[
 F(v,T) = 0.
 \]
In this section we prove that $(0,T_*)$ is a bifurcation point of $\Xi$ for the zero level set of the operator $F$. 

 \medskip
 Proposition  \ref{kernel.HT} ensures that the kernel of the
 operator $H_{T_*}$ is finite-dimensional and it equals  $V_{j_1} \oplus
 \cdots \oplus V_{j_l}$.
 Let $Q$ be the projection operator onto the image of $H_{T_*}$ and $Q
 \circ F$ the composition of operators $F$ and $Q$. We write a function
 $v \in \mathcal \mathcal{C}^{2,\alpha}_{\textnormal{even},0} (\mathbb{R}/2\pi
 \mathbb{Z})$ as $v = v^{\|}\, + v^{\bot}$ with $v^{\|} \in \textnormal{Ker}H_{T_*}$ and
 $v^{\bot} \in (\textnormal{Ker}H_{T_*})^{\bot}$.
 The next result (that represent the classical Lyapunov-Schmidt reduction for our problem) follows from the implicit function Theorem:

 \begin{proposition}
 \label{lyapunov}
 For all $v^{\|} \in \textnormal{Ker}H_{T_*}$ whose norm is small
 enough and for all $T$ sufficiently close to $T_*$ there exists
 a unique function $v^{\bot} = v^{\bot}(v^{\|},T)$ such that
 \[
 Q \circ F \left(v^{\|}\, + v^{\bot}, T\right) = 0.
 \]
 \end{proposition}
 \begin{proof} Define the operator $J$ as follows:
 \[
 J(v^{\|}, v^{\bot},T) = Q \circ F \left(v^{\|}\, + v^{\bot}, T\right)
 \]
 from $\textnormal{Ker}H_{T_*} \times (\textnormal{Ker}H_{T_*})^{\bot} \times (0,+\infty)$ into the image of $H_{T_*}$. By Proposition \ref{kernel.HT} the implicit function theorem applies to get the existence of a unique function
 \[
 v^{\bot}(v^{\|},T) \in (\textnormal{Ker}H_{T_*})^{\bot}
 \]
 smoothly depending on $v^{\|}$ and $T$ in a neighborhood of $(0,T_*)$ such that
 \[
 J(v^{\|}, v^{\bot}(v^{\|},T),T) = 0.
 \]
 This completes the proof of the proposition. \hfill $\Box$
 \end{proof}

 \medskip

 Now we can define the operator
 \[
 G(v^{\|},T)=
 (I - Q) \circ F \left(v^{\|}\, + v^{\bot}(v^{\|},T), T\right) = 0.
 \]
 where $I$ is the identity operator and $v^{\bot}(v^{\|},T)$ is the function given by Proposition \ref{kernel.HT}. $G$ is a finite-dimensional operator from $\textnormal{Ker}H_{T_*} \times (0,+\infty)$ into the space orthogonal to the image of $H_{T_*}$. We remark that our main theorem \ref{maintheorem} will be proved if we show that $(0,T_*)$ is a bifurcation point for the zero level set of $G$. In fact, it is easy to prove that the curve
 \[
\Gamma =  \{(v^{\|},T) \in \textnormal{Ker}H_{T_*} \times
(0,+\infty) \quad : \quad v^{\|} =0\}
 \]
 is a solution of $ G(v^{\|},T)=0$ with $v^{\bot}(0,T) =0$. Then, the fact that $(0,T_*)$ is a bifurcation point  of $\Gamma$ for the zero level set of $G$ means that in every neighborhood of
 $(0, T_*)$ in $\textnormal{Ker}H_{T_*} \times (0,+\infty)$ contains solutions of
 the equation $ G(v^{\|},T)=0$ which are not in $\Gamma$, i.e. there exists a sequence $(v_i^{\|}, T_i) \in \textnormal{Ker}H_{T_*} \times (0,+\infty)$ with $v_i^{\|} \neq 0$ such that $ G(v_i^{\|},T_i)=0$. Hence
 \[
 Q \circ F \left(v_i^{\|}\, + v^{\bot}(v_i^{\|},T_i), T_i\right) = 0
 \]
 and
 \[
 (I - Q) \circ F \left(v_i^{\|}\, + v^{\bot}(v_i^{\|},T_i), T_i\right) = 0
 \]
 that imply
 \[
 F \left(v_i^{\|}\, + v^{\bot}(v_i^{\|},T_i), T_i\right)= 0
 \]
and $v_i : = v_i^{\|}\, + v^{\bot}(v_i^{\|},T_i) \neq 0$.

\medskip

Let us prove that $(0,T_*)$ is a bifurcation point of $\Gamma$ for the zero level set of $G$.
 We start by recalling a useful result about bifurcation (see \cite{kielhofer} and \cite{smoller} for details). Let $L$ be an operator on $\mathbb{B}_1 \times \Lambda$ into
 $\mathbb{B}_2$, where $\mathbb{B}_1$ and $\mathbb{B}_2$ are Banach
 spaces (or subspaces) and $\Lambda$ is an interval of $\mathbb{R}$.
 Thus suppose that $\Gamma = (x(s),s)$ is a curve of solutions of the
 equation $L(x,s)=0$. Let $(x_0,s_0) = (x(s_0), s_0)$ be an interior
 point on this curve with the property that every neighborhood of
 $(x_0,s_0)$ in $\mathbb{B}_1 \times \Lambda$ contains solutions of
 the equation $L(x,s)=0$ which are not in $\Gamma$, i.e. it is a \textit{bifurcation point} of $\Gamma$ for the zero level set of $L$. In our case $\mathbb{B}_1 = C^{2,\alpha}_{\textnormal{even},0} (\mathbb{R}/2\pi \mathbb{Z})$, $\Lambda  = (0,+\infty)$, $\mathbb{B}_2 = C^{1,\alpha}_{\textnormal{even},0} (\mathbb{R}/2\pi \mathbb{Z})$ and $x(s) = 0$ for all $s$.
 A necessary condition for bifurcation at $(0,s_0)$ is that $0$ is an isolated
 eigenvalue of finite algebraic multiplicity, say $l$, of the
 operator obtained by linearizing $L$ with respect to $x$ at
 $(0,s_0)$, which can be denoted by $D_{x}L(0,s_0)$. It is crucial to
 know how the eigenvalue $0$ of $D_{x}L(0,s_0)$ changes when $s$
 varies in a neighborhood of $s_0$. It is possible to show (see
 \cite{kato}) that the generalized eigenspace
 $E_{s_0}$ of the eigenvalue $0$ of $D_{x}L(0,s_0)$ having dimension
 $l$ is perturbed to an invariant space $E_{s}$ of $D_{x}L(0,s)$ of
 dimension $l$ too, and all perturbed eigenvalues near $0$ (the
 so-called $0$-group) are eigenvalues of the finite-dimensional
 operator $D_{x}L(0,s)$ restricted to the $l$-dimensional invariant
 space $E_s$. Moreover the eigenvalues in that $0$-group depend
 continuously on $s$. Let us give the definition of {\it odd crossing number}:

 \begin{definition}
 We set $\Theta(s)$ to be equal to $1$ if there are no
 negative real eigenvalues in the $0$-group of $D_{x}L(0,s).$
 Otherwise
 \[
 \Theta(s) = (-1)^{l_1 + \cdots + l_h}
 \]
 if $\mu_1, \ldots, \mu_h$ are all the negative real eigenvalues of
 the $0$-group having algebraic multiplicity $l_1,\ldots,l_h$,
 respectively. If $D_{x}L(0,s)$ is regular in a neighborhood of $s_0$
 (naturally except in the point $s_0$) and $\Theta(s)$ changes the
 sign at $s_0$ then $D_{x}L(0,s)$ is said to have an odd crossing
 number at $s_0$.
 \end{definition}

 In presence of an odd crossing number,  a standard result
 known as the Krasnosel'skii Bifurcation Theorem (see \cite{kielhofer} for
 the proof) applies:

 \begin{theorem}
 If $D_{x}L(0,s)$ has an odd crossing number at $s_0$, then $(0,s_0)$
 is a bifurcation point for $L(x,s) =0$ with respect to the curve
 $\{(0,s)\,\, |\,\, s\ \textnormal{in a neighborhood of}\ s_0\}$.
 \end{theorem}

The fact that $(0,T_*)$ is a bifurcation point for the operator $G$ follows then from the Krasnosel'skii Bifurcation Theorem and the following:

 \begin{proposition}
 $D_{v^{\|}}G(0,T)$ has an odd crossing number at $T_*$.
 \end{proposition}

 \begin{proof} We observe that we can write
 \[
 v^{\|} = \sum_{i=1}^l a_{k_i}\, \cos(k_i\, t)
 \]
 where $1 = k_1 < \cdots < k_l$.
 It is clear, from the definition of $G$, that $D_{v^{\|}}G(0,T)$ preserves
 the eigenspaces, and
 \[
 D_{v^{\|}}G(0,T) = H_T |_{V_{j_1} \oplus \cdots \oplus V_{j_l}}.
 \]
 Then the $0$-group of eigenvalues is given by $\sigma_{j_1}(T),\ldots,
 \sigma_{j_l}(T)$, where $\sigma_{j_1}(T) = \sigma(T)$. For $T=T_*$ they
 are all equal to $0$. Moreover, by the proposition \ref{kernel.HT} only
 $\sigma_{j_1}(T)$ changes sign at $T_*$, and the corresponding eigenspace
 has dimension $1$. This means that $D_{v^{\|}}G(0,T)$ has a crossing
 number at $T_*$ and completes the proof of the proposition.  \hfill $\Box$
 \end{proof}

\end{document}